\documentclass{amsart}

\usepackage{amsfonts,amssymb,amsmath,amsthm,latexsym,booktabs,lscape,color,enumerate, amscd}

\theoremstyle{definition}
\newtheorem{definition}{Definition}[section]
\newtheorem{remark}[definition]{Remark}
\newtheorem{example}[definition]{Example}

\theoremstyle{plain}
\newtheorem{lemma}[definition]{Lemma}
\newtheorem{proposition}[definition]{Proposition}
\newtheorem{theorem}[definition]{Theorem}
\newtheorem{corollary}[definition]{Corollary}

\newtheorem{notation}[definition]{Notation}
\newtheorem{defis}[definition]{Definitions}
\newtheorem{quest}[definition]{Question}

\newcommand{\Z}{\mathrm{Z}}
\newcommand{\M}{\mathrm{M}}
\newcommand{\T}{\mathcal{T}}

\newcommand{\A}{\mathcal{A}}

\newcommand{\C}{\mathbb{C}}
\newcommand{\Tri}{\mathrm{Trian \,}}
\newcommand{\Id}{\mathrm{Id}}
\newcommand{\Nil}{\mathrm{Nil}}
\newcommand{\Lann}{\mathrm{Lann}}
\newcommand{\Rann}{\mathrm{Rann}}
\newcommand{\Imm}{\mathrm{Im}}

\usepackage{color}



\begin{document}

\title[Automorphisms, $\sigma$-biderivations and $\sigma$-commuting maps of TA's]{Automorphisms, $\sigma$-biderivations and $\sigma$-commuting maps of triangular algebras}

\author[Mart\'in Gonz\'alez, Repka, S\'anchez-Ortega]{C\'andido Mart\'in Gonz\'alez$^{1}$, Joe Repka$^{2}$, Juana S\'anchez-Ortega$^{3}$}

\address[1]{Departamento de \'Algebra, Geometr\'ia y Topolog\'ia, Universidad de M\'alaga, M\'alaga, Spain}
\address[2]{University of Toronto\\ Toronto, ON, Canada}
\address[3]{Department of Mathematics and Applied Mathematics. University of Cape Town \\ Cape Town, South Africa}
\email{candido@apncs.cie.uma.es}
\email{repka@math.toronto.edu}
\email{juana.sanchez-ortega@uct.ac.za}

\maketitle


\begin{abstract}
Let $\A$ be an algebra and $\sigma$ an automorphism of $\A$. A linear map $d$ of $\A$ is called a $\sigma$-derivation of $\A$ if $d(xy) = d(x)y + \sigma(x)d(y)$, for all $x, y \in \A$. A bilinear map $D: \A \times \A \to \A$ is said to be a $\sigma$-biderivation of $\A$ if it is a $\sigma$-derivation in each component. An additive map $\Theta$ of $\A$ is $\sigma$-commuting if it satisfies $\Theta(x)x - \sigma(x)\Theta(x) = 0$, 
for all $x \in \A$. In this paper, we introduce the notions of inner and extremal $\sigma$-biderivations and of proper $\sigma$-commuting maps. One of our main results states that, under certain assumptions, every $\sigma$-biderivation of a triangular algebras is the sum of an extremal $\sigma$-biderivation and an inner $\sigma$-biderivation. Sufficient conditions are provided on a triangular algebra for all of its $\sigma$-biderivations 
(respectively, $\sigma$-commuting maps) to be inner (respectively, proper). A precise description of $\sigma$-commuting maps of triangular algebras is also given. A new class of automorphisms of triangular algebras is introduced and precisely described. We provide many classes of triangular algebras whose automorphisms can be precisely described.

\end{abstract}


\section{Introduction}

Triangular algebras were introduced by Chase \cite{Cha} in the early 1960s. He ended up with these structures 
in the course of his study of the asymmetric behavior of semi-hereditary rings; he provided an example of a left semi-hereditary ring which is not right semi-hereditary. Since their introduction, triangular algebras have played an important role in the development of ring theory. In 1966, Harada \cite{H} characterized hereditary semi-primary rings by using triangular algebras; in his paper, triangular algebras are named generalized triangular matrix rings. Later on, Haghany and Varadarajan \cite{HV} studied many properties of triangular algebras; they used the terminology of formal triangular matrix rings. 

Derivations are a fundamental notion in mathematics. They have been objects of study by many well-known mathematicians. In the middle/late 1990s, several authors undertook the study of derivations and related maps over some particular families of triangular algebras (see for e.g. \cite{Chr, CM, FM, Jn, Zh} and references therein). Motivated by those works Cheung \cite{Ch1} initiated, in his thesis, the study of linear maps of (abstract) triangular algebras. He described automorphisms, derivations, commuting maps and Lie derivations of triangular algebras. Cheung's work has inspired several authors to investigate many distinct maps of triangular algebras.

The present paper is devoted to the study of automorphisms and the so-called $\sigma$-biderivations and $\sigma$-commuting maps of triangular algebras. Our motivation arises from recent developments in the theory of maps of triangular algebras; more concretely, we could say that our interest was sparked by the work of Cheung \cite{Ch2} on commuting maps of triangular algebras, the paper of Benkovi\v{c} \cite{Bk2} on biderivations of triangular algebras, and the study of $\sigma$-biderivations and $\sigma$-commuting maps of nest algebras due to Yu and Zhang \cite{YZ}. 

Because the group of automorphisms is an important tool for understanding the underlying algebraic structure, it has been most thoroughly investigated. The groups of automorphisms of triangular algebras have been investigated by many authors; see, for example, \cite{AW, BK, Ch1, Co, CX, Jn, KDW, K, L1, L2}. 

The study of biderivations of rings and algebras has a rich history and is still an active area of research. Also, biderivations have already been shown to be of utility 
in connection with distinct areas; for example, Skosyrskii \cite{Sk} used them to investigate 
noncommutative Jordan algebras, while Farkas and Letzter \cite{FL} treated them in their study of Poisson algebras. 

Let $\A$ be an algebra over a unital commutative ring of scalars; a bilinear map $D: \A \times \A \to \A$ is said to be a {\bf biderivation of} $\A$ if it is a derivation in each argument. The map $\Delta_\lambda: (x, y) \mapsto \lambda [x, y]$ is an example of a biderivation, provided that $\lambda$ lies in the center $\Z(\A)$ of $\A$. Biderivations of this form  
are called {\bf inner biderivations}; here, $[x, y]$ stands for the commutator $xy - yx$.  

Notice that if $d$ is a derivation of a commutative algebra, then the map $(x, y) \mapsto d(x)d(y)$ is a biderivation. In the noncommutative case, it happens quite often that all biderivations are inner. The classical problem under study is to determine whether every biderivation of a noncommutative algebra is inner. 

In 1993, Bre\v{s}ar et al. \cite{BrMM} proved that every biderivation of a noncommutative prime ring $R$ is inner. One year after, Bre\v{s}ar \cite{Br1} extended the result above to semiprime rings. Those results have been shown to be very useful in the study of the so-called commuting maps that, as we will point out below, are closely related to biderivations.

Recently, motivated by Cheung's work, Zhang et al. \cite[Theorem 2.1]{ZhFLW} proved that, under some mild conditions, every biderivation of a nest algebra is inner. Later on, Zhao et al. \cite{ZYW} proved, 
using the results of \cite{ZhFLW}, that every biderivation of an upper triangular matrix algebra is a sum of an inner biderivation and an element of a particular class of biderivations. Benkovi\v{c} \cite{Bk2} was the first author to study biderivations of (abstract) triangular algebras. He determined the conditions which need to be imposed on a triangular algebra to ensure that all its biderivation are inner. 

It turns out that biderivations are closely connected to the thoroughly studied commuting (additive) maps. A map $\Theta$ of an algebra $\A$ into itself is said to be {\bf commuting} if $\Theta(x)$ commutes with $x$, for every $x \in \A$. The usual goal when dealing with a commuting map is to provide a precise description of its form. It is straightforward to check that the identity map and any central map (a map having its image in the center of the algebra) are examples of commuting maps. Moreover, the sum and the pointwise product of commuting maps produce commuting maps. For example, it is easy to check that the following map
\begin{equation*} \label{propercomm}
\Theta (x) = \lambda x + \Omega(x), \quad \, \forall \, x\in \A,
\end{equation*} 
is commuting, for any $\lambda \in \Z(\A)$
and any choice of central map $\Omega$ of $\A$. We will refer to commuting maps of this form as {\bf proper commuting maps}.

The first important result on commuting maps dates back to 1957 and it is due to Posner \cite{Po}. Posner's theorem states that the existence of a nonzero commuting derivation of a prime ring implies the commutativity of the ring. Notice that Posner's theorem can be reformulated as follows: 

\smallskip

{\bf Theorem.} {\it If $d$ is a commuting derivation of a noncommutative prime ring $R$, then $d = 0$.}

\smallskip 

osner's theorem has been generalized by a number of authors in many different ways. We refer the reader to the well-written survey \cite{Br3}, where the development of the theory of commuting maps and related maps is presented; see also \cite[Section 3]{Br3} for applications of biderivations and commuting maps to other 
areas. Let us point out here that one of the main applications of commuting maps can be found in their connection with several conjectures on Lie maps of associative rings, formulated by Herstein \cite{He} in 1961.

Let us now come back to the relationship between biderivations and commuting maps; notice that if $\Theta$ is a commuting map of an algebra $\A$, then the map $D_\Theta: \A \times \A \to \A$ given by $D_\Theta(x, y) = [\Theta(x), y]$  
is a biderivation of $\A$. Moreover, suppose that there exists $\lambda \in \Z(\A)$ such that $[\Theta(x), y] = \lambda[x, y]$, 
for all $x, y \in \A$; that is to say that $D_\Theta$ is an inner biderivation. Then it follows that $\Omega_\Theta(x):= \Theta(x) - \lambda x \in \Z(\A)$, for every $x \in \A$; in other words, $\Theta$ is a proper commuting map.
Therefore, in order to show that every commuting map is proper, it is enough to prove that  
every biderivation is inner.

Let $\sigma$ be an automorphism of an algebra $\A$. Recall that a linear map $d$ of $\A$ is called a {\bf $\sigma$-derivation} if $d(xy) = \sigma(x) d(y) + d(x) y$, 
for all $x, y \in \A$. A {\bf $\sigma$-biderivation} is a bilinear map which is a $\sigma$-derivation in each argument.
 
The structure of $\sigma$-biderivations of prime rings was investigated by Bre\v{s}ar \cite{Br2}. As an application, he characterized additive maps $f$ of a prime ring satisfying $f(x)x = \sigma(x)f(x)$, 
for all $x$ in the ring. We will refer to such 
maps as {\bf $\sigma$-commuting maps}. In the context of triangular algebras, $\sigma$-biderivations and $\sigma$-commuting maps of nest algebras have been studied by Yu and Zhang \cite{YZ} in 2007. 

\smallskip

The paper is organized as follows: in Section 2 we gather together basic definitions and elementary properties. Section 3 is devoted to describing broad classes of triangular algebras whose automorphisms are partible in the sense of  \cite[Definition 5.1.6]{Ch1}. We call a triangular algebra partible if all its automorphisms are partible. In Section 4 we introduce and describe precisely a new class of automorphisms of triangular algebras, the so-called {\bf bimodule-preserving}. Section 5 deals with the study of $\sigma$-biderivations of partible triangular algebras, while $\sigma$-commuting maps are investigated in Section 6. 


\section{Preliminaries}

Throughout the paper we consider unital associative algebras, and we assume without further mention that all of them are algebras over a fixed commutative unital ring of scalars $R$. In the subsequent subsections, we recall some definitions and basic results, and introduce some notation.

\subsection{Triangular algebras} \label{preli}

Let $A$ and $B$ be unital associative algebras and $M$ a nonzero $(A, B)$-bimodule. The following set becomes an associative algebra under the usual matrix operations:
\[
\T = \Tri(A, M, B) = \left(
\begin{array}{cc}
A & M  \\
  & B 
\end{array}
\right) = 
\left\{
\left(
\begin{array}{cc}
a & m  \\
  & b 
\end{array}
\right): \, a \in A, m\in M, b \in B 
\right\}.
\]
An algebra $\T$ is called a {\bf triangular algebra} if there exist algebras $A$, $B$ and a nonzero $(A, B)$-bimodule $M$ such that $\T$ is isomorphic to $\Tri(A, M, B)$. 

Given $\T = \Tri(A, M, B)$, 
a triangular algebra,  
let us denote by $\pi_A$, $\pi_B$ the 
two natural projections from $\T$ onto $A$, $B$, respectively defined as follows: 
\[
\pi_A: 
\left(
\begin{array}{cc}
a & m  \\
  & b 
\end{array}
\right) \mapsto a, 
\quad \quad 
\pi_B: 
\left(
\begin{array}{cc}
a & m  \\
  & b 
\end{array}
\right) \mapsto b, \quad  \mbox{ for all } \, \left(
\begin{array}{cc}
a & m  \\
  & b 
\end{array}
\right) \in \T.
\]
The center $\Z(\T)$ of $\T$ was computed in \cite{Ch1}  
(see also \cite[Proposition 3]{Ch2}); 
it is the following set:
\[
\Z(\T) = \left\{
\left(
\begin{array}{cc}
a & 0  \\
  & b 
\end{array}
\right) \in \T: \, a \in \Z(A), \, b \in \Z(B) \, \mbox{ and } \, am = mb, \, \,  
\mbox{ for all } m\in M
\right\}.
\]
At this point, it is worth introducing the following two sets:
\[
L := \{a \in A: aM = 0\}, \quad \quad R := \{b \in B: Mb = 0\}.
\] 
It is easy to see that $L$ and $R$ are ideals of $A$ and $B$, respectively. Moreover, $LM = MR = 0$. Set $\Z := \Z(\T)$; by an abuse of notation we denote by $\bar{a}$ and $\bar{b}$ the classes of the elements $a \in \pi_A(\Z)$,  $b \in \pi_B(\Z)$ in $\pi_A(\Z)/ (\pi_A(\Z) \cap L)$ and $\pi_B(\Z)/ (\pi_B(\Z) \cap R)$, respectively. Then the map $\tau: \pi_A(\Z)/ (\pi_A(\Z) \cap L) \to  \pi_B(\Z)/ (\pi_B(\Z) \cap R)$ given by $\tau(\bar{a}):= \bar{b}$, where $b = \pi_B(x)$ for some $x \in \T$ with $\pi_A(x) = a$, is well-defined and an algebra isomorphism. In particular, if the bimodule $M$ is faithful as a left $A$-module and also as a right $B$-module, we have that $L = R = 0$ and therefore we can conclude that there exists a unique algebra isomorphism $\tau:\pi_A(\Z) \to \pi_B(\Z)$ satisfying that $am = m\tau(a)$, for all $m \in M$ and $a \in \pi_A(\Z)$. 

Upper triangular matrix algebras, block upper triangular matrix algebras, triangular Banach algebras and nest algebras constitute the most important examples of triangular algebras. (See, for example, \cite{Ch1}.)

\begin{notation}\label{notation}
{\rm 
Let $\T = \Tri(A, M, B)$ be a triangular algebra. 
We will write $1_A$, $1_B$ to denote the units of the algebras $A$, $B$, respectively. The unit of $\T$ is the element: 
\[
1 := \left(
\begin{array}{cl}
1_A & 0  \\
    & 1_B
\end{array}
\right). 
\]
It is straightforward to check that the following elements are orthogonal idempotents of $\T$.
\[
p := 
\left(
\begin{array}{cc}
1_A & 0  \\
    & 0
\end{array}
\right), 
\quad \quad
q := 1 - p = 
\left(
\begin{array}{cl}
0 & 0  \\
  & 1_B
\end{array}
\right) \in \T. 
\]
Hence, we can consider the Peirce decomposition of $\T$ associated to the idempotent $p$. In other words, we can write $\T = p\T p \oplus p \T q \oplus q \T q$. Note that $p\T p$, $q \T q$ are subalgebras of $\T$ isomorphic to $A$, $B$, respectively; while $p \T q$ is a $(p \T p, q \T q)$-bimodule isomorphic to $M$. To ease the notation in what follows, we will identify $p\T p$, $q \T q$, $p \T q$ with $A$, $B$, $M$, respectively. Thus, any element of $\T$ can be expressed as:
\[
x = \left(
\begin{array}{cc}
a & m  \\
  & b 
\end{array}
\right) = pap + pmq + qbq = a + m + b,
\]
by using the identification above. 
}
\end{notation}

\subsection{Derivations and related maps}
In this subsection, we collect all the definitions of the maps that will be used in
the paper. Although we have already introduced some of those maps in the introductory section, we will include their definition here for the sake of completeness. 

\begin{defis}
{\rm
Let $\A$ be an algebra and $\sigma$ an automorphism of $\A$. Let us denote by $\Id_\A$ the identity map on $\A$.
\begin{itemize}
\item A linear map $d: \A \to \A$ is called a {\bf derivation of} $\A$ if it satisfies 
\[
d(xy) = d(x)y + xd(y), \quad \forall \, x, y \in \A.
\]
\item A linear map $d: \A \to \A$ is called a {\bf $\sigma$-derivation of} $\A$ if it verifies 
\[
d(xy) = d(x)y + \sigma(x)d(y), 
\quad \forall \, x, y \in \A.
\]
Note that every derivation is an $\Id_\A$-derivation. In the literature (see, for example, \cite{HW, YaZh}), some authors have used the terminology of skew-derivations for $\sigma$-derivations.  
\item A bilinear map $D: \A \times \A \to \A$ is said to be a {\bf biderivation of} $\A$ if it is a derivation in each argument; that is to say that for every $y\in A$, the maps $x \mapsto D(x, y)$ and $x \mapsto D(y, x)$ are derivations of $\A$. In other words:
\begin{center}
$D(xy, z) = x D(y, z) + D(x, z)y, \quad  D(x, yz) = y D(x, z) + D(x, y)z$,
\end{center}
for all $x, y, z\in \A$. The map $D$ is called a {\bf $\sigma$-biderivation of} $\A$ provided that $D$ is a $\sigma$-derivation in each argument.
\item Suppose that $\A$ is noncommutative and take $\lambda \in \Z(\A)$. The bilinear map $\Delta_\lambda: \A \times \A \to \A$ given by 
\[
\Delta_\lambda (x, y) = \lambda [x, y], \quad 
\forall \, x, y \in \A
\] 
is an example of a biderivation. Biderivations of the form above are called {\bf inner biderivations}.
\item Given $x_0 \in \A$ such that $x_0 \notin \Z(\A)$; suppose that $x_0$ satisfies 
$[[x, y], x_0]Ê= 0$ for all $x, y \in \A$. It was proved in \cite[Remark 4.4]{Bk2} that the bilinear map $\psi_{x_0}: \A \times \A \to \A$ given by 
\[
\psi_{x_0} (x, y) = [x, [y, x_0]], \quad \forall \, x, y \in \A
\]
is a biderivation. Biderivations of this form appear first in \cite{Bk2}; they were named {\bf extremal biderivations}.
\end{itemize}
}
\end{defis}

In this paper, we introduce the notions of inner $\sigma$-biderivations and extremal $\sigma$-biderivations (see Definitions \ref{sigmainn} and \ref{defiextbi} below).

\begin{example} \label{bide}
Let $\A$ be an algebra and $\sigma$ a nontrivial automorphism of $\A$. It is easy to see that the map $d := \Id_\A - \sigma$ is a $\sigma$-derivation of $\A$, although $d$ is not, in general, a derivation of $\A$. For example, let $K$ be a field of characteristic not $2$, and take the algebra of polynomials $K[x]$ 
and $\sigma$ the automorphism of $K[x]$ which maps $x$ to $-x$. Then the $\sigma$-derivation $d$ of $K[x]$, defined as before, satisfies 
\begin{align*} 
& d(x^2) = x^2 - \sigma(x^2) = x^2 - \sigma(x) \sigma(x) = x^2 - (-x)(-x) = 0,
\\
& d(x)x + xd(x) = 4x^2, 
\end{align*}
since $d(x) = x - \sigma(x) = 2x$. This shows that $d$ is not a derivation of $K[x]$. On the other hand, the map $D(x, y) = d(x)d(y)$ is a $\sigma$-biderivation, since $K[x]$ is a commutative algebra. But $D$ is not a biderivation, since $D(x^2, x) = 0$ and $D(x, x)x + xD(x, x) = 8x^3$. 

Notice that these examples also work for the finite-dimensional algebra of polynomials of degree $\leq N$ (i.e., the polynomial algebra $K[x]$ modulo the ideal generated by $x^{N+1}$), since $\sigma$ preserves degrees.
\end{example}

\begin{defis}
{\rm
Let $\A$ be an algebra, $\sigma$ an automorphism of $\A$ and $\Theta$ a linear map 
$\A \to \A$.
\begin{itemize}
\item The map $\Theta$ is a {\bf commuting map}  
$\A \to \A$ if $\Theta(x)$ commutes with $x$,
for every $x \in \A$, i.e., $[\Theta(x), x] = 0$, 
for all $x\in \A$. 
\item $\Theta$ is called {\bf $\sigma$-commuting} if it verifies 
$\Theta(x)x = \sigma(x)\Theta(x)$,  
for all $x \in \A$. In particular, in the case 
where $\sigma = \Id_A$,
the notion of a commuting map is recovered. A this point, it is worth mentioning that the concept of a skew-commuting map has a different meaning (see, for example, \cite{Br01}).
\end{itemize}
}
\end{defis}

\begin{example}
Let $\T_2(\C)$ be the algebra of 
$2 \times 2$ upper triangular matrices over the complex numbers. The map $\sigma : \T_2(\C) \to \T_2(\C)$ given by
\[
\sigma \left(
\begin{array}{cc}
a & b  \\
  & c 
\end{array}
\right) := \left(
\begin{array}{cr}
a & -b  \\
  & c 
\end{array}
\right), \quad \mbox{for all} \, \, \left(
\begin{array}{cc}
a & b  \\
  & c 
\end{array}
\right) \in \T_2(\C),
\]
is an automorphism of $\T_2(\C)$. It is straightforward to show that the linear map $\Theta$ of $\T_2(\C)$ given by 
\[
\Theta \left(
\begin{array}{cc}
a & b  \\
  & c 
\end{array}
\right) := \left(
\begin{array}{cr}
a &  b  \\
  & -c 
\end{array}
\right), \quad \mbox{for all} \, \, \left(
\begin{array}{cc}
a & b  \\
  & c 
\end{array}
\right) \in \T_2(\C),
\]
is a $\sigma$-commuting map of $\T_2(\C)$; however, it is not commuting since, for example, for $x = \left(
\begin{array}{cc}
0 & 1  \\
  & 1 
\end{array}
\right)$, we have that $[x, \Theta(x)] = \left(
\begin{array}{cr}
0 & -2  \\
  &  0 
\end{array} 
\right) \neq 0$.
\end{example}

\begin{remark}
When dealing with a triangular algebra $\T = \Tri(A, M, B)$, it is quite common to assume that the bimodule $M$ is faithful as a left $A$-module and also as a right $B$-module. At this point, it is worth mentioning 
that we can indeed restrict our attention to triangular algebras having the faithfulness condition above.  In fact, let $\T = \Tri(A, M, B)$ be a triangular algebra.  Notice that $\Tri(A/L, M, B/R)$ is a triangular algebra and $M$ is faithful as a left $A/L$-module and as a right $B/R$-module, where the annihilators $L$ and $R$ are the ideals of $A$ and $B$, respectively, defined in Subsection \ref{preli}.  

It is easy to check that $L\oplus R$ is an ideal of $\Tri(A, M, B)$ and that 
\[
\Tri(A/L, M, B/R) \cong
\Tri(A, M, B)/(L \oplus R).
\]
\end{remark}

\section{Partible automorphisms of triangular algebras} 

Let $\T = \Tri(A, M, B)$ be a triangular algebra. In what follows, we will assume that the bimodule $M$ is faithful as a left $A$-module and also as a right $B$-module, although these assumptions might not always be necessary.

Automorphisms of triangular algebras were studied by Cheung \cite[Chapter 5]{Ch1}. He introduced and characterized the following class of automorphisms: 

\begin{definition} \cite[Definition 5.1.6]{Ch1} \label{autpart}
{\rm
An automorphism $\sigma$ of a triangular algebra $\T = \Tri(A, M, B)$ is said to be {\bf partible with respect to $A$, $M$ and $B$}, if it can be written as $\sigma = \phi_z \bar{\sigma}$, where $z\in \T$, $\phi_z$ is the inner automorphism $\phi_z(x) = z^{-1}xz$, associated to $z$, and $\bar{\sigma}$ is an automorphism of $\T$ satisfying that $\bar{\sigma}(A) = A$, $\bar{\sigma}(M) = M$ and $\bar{\sigma}(B) = B$.
}
\end{definition}

As Cheung pointed out inner automorphisms are all partible. 

\begin{remark}
Let $\T = \Tri(A, M, B)$ be a triangular algebra. It is not difficult to prove that every automorphism $\sigma$ of $\T$ satisfying that ${\sigma}(A) = A$, ${\sigma}(M) = M$ and ${\sigma}(B) = B$ can be written as follows:
\begin{equation} \label{autom}
\sigma \left(
\begin{array}{cc}
a & m  \\
  & b 
\end{array}
\right) = 
\left(
\begin{array}{ccc}
f_\sigma(a) &&  \nu_\sigma(m)  
\\
&&
\\
  && g_\sigma(b) 
\end{array}
\right),
\end{equation}
where $f_\sigma$, $g_\sigma$ are automorphisms of $A$, $B$, respectively, and $\nu_\sigma: M \to M$ is a linear bijective map which satisfies $\nu_\sigma(am) = f_\sigma(a)\nu_\sigma(m)$ and $\nu_\sigma(mb) = \nu_\sigma(m) g_\sigma(b)$, 
for all $a \in A$, $b \in B$, $m \in M$. 
\end{remark}

In this paper we will restrict our attention to the following class of triangular algebras:

\begin{definition}
A triangular algebra $\T$ is said to be {\bf partible} if every automorphism of $\T$ is partible. 
\end{definition}

See \cite[Example 5.2.4]{Ch1} for an example of a triangular algebra which is not partible.

\medskip

\begin{proposition} 
Upper triangular matrix algebras and nest algebras are partible.
\end{proposition}

\begin{proof}
It follows from the fact that inner automorphisms are partible by taking into account the fact that automorphisms of upper triangular matrix algebras and nest algebras are all inner. For more details, see \cite{K} for upper triangular matrix algebras and \cite{L1, L2} for nest algebras.
\end{proof}

Cheung \cite[Theorem 5.3.2]{Ch1} gave sufficient conditions on a triangular algebra $\T$ to assert that $\T$ is partible. 
In what follows, we will make use of Cheung's result to provide many classes of partible triangular algebras. Let us start by introducing the following condition based on idempotents.

\begin{definition}
We will say that an algebra $\A$ satisfies {\bf Condition (I)} if any idempotent $e \in \A$ such that $e\A(1-e) = 0$ 
must satisfy $(1 - e) \A e = 0$. 
\end{definition}

The next result collects some examples of algebras satisfying Condition (I).

\begin{proposition}
Let $\A$ be an algebra and consider the following properties:
\begin{itemize}
\item[\rm{(i)}] $\A$ is commutative.
\item[\rm{(ii)}] Every idempotent of $\A$ is central. 
\item[\rm{(iii)}] $\A$ satisfies {\rm Condition (I)}.
\end{itemize}
Then {\rm (i)}Ê$\, \, \Rightarrow$ {\rm (ii)} $\Rightarrow$ {\rm (iii)}.
\end{proposition}

\begin{proof}
Straightforward.
\end{proof}

Next, we will show that non-degenerate algebras, also called semiprime algebras, satisfy Condition (I). Recall that an algebra $\A$ is said to be {\bf non-degenerate} if, for any $a \in \A$,
$a\A a = 0$ implies $a = 0$.

\begin{proposition} \label{nondegeI}
Every non-degenerate algebra satisfies {\rm Condition (I)}.
\end{proposition}

\begin{proof}
Let $\A$ be a non-degenerate algebra and $e\in \A$ an idempotent such that $e\A(1-e) = 0$. Then we have that
$((1 - e) \A e )
\A 
((1 - e) \A e ) = 0$, which implies that $(1 - e) \A e  = 0$, as desired.
\end{proof}

\begin{theorem} \label{(I)}
Let $\T = \Tri(A, M, B)$ be a triangular algebra such that either $A$ or $B$ satisfies {\rm Condition (I)}. Then $\T$ is partible.
\end{theorem}

\begin{proof}
It follows from \cite[Theorem 5.3.2 (I) or (II)]{Ch1}.
\end{proof}

The following are direct consequences of the theorem above.  

\begin{corollary} \label{nondegpartible}
Let $\T = \Tri(A, M, B)$ be a triangular algebra such that either $A$ or $B$ is non-degenerate. Then $\T$ is partible.
\end{corollary}

\begin{corollary}
Let $\T = \Tri(A, M, B)$ be a triangular algebra such that either $A$ or $B$ is commutative. Then $\T$ is partible.
\end{corollary}

\begin{corollary}
Let $\T = \Tri(A, M, B)$ be a triangular algebra such that all the idempotents of $A$ (respectively, $B$) are central. Then $\T$ is partible.
\end{corollary}

The following result provides a sufficient condition for an automorphism to be partible. 

\begin{theorem} \label{key}
Let $\T = \Tri(A, M, B)$ be a triangular algebra and $\sigma$ an automorphism of $\T$ such that $\sigma(M) = M$. Then $\sigma$ is partible.
\end{theorem}

We first prove a preliminary lemma.

\begin{lemma} \label{lann&rann}
Let $\T = \Tri(A, M, B)$ be a triangular algebra. We write $\Lann_{\T}(M)$ and $\Rann_{\T}(M)$ for the left annihilator of $M$ in $\T$ and 
the right annihilator of $M$ in $\T$,  respectively.
Then the following conditions hold:
\begin{enumerate}
\item[{\rm (i)}] $\Lann_\T (M) = M \oplus B$,
\item[{\rm (ii)}] $\Rann_\T (M) = A \oplus M$.
\end{enumerate}
\end{lemma}

\begin{proof}
(i). It is clear that $M \oplus B \subseteq \Lann_\T (M)$. To show the other containment, take an element $x = a + m + b \in \Lann_\T (M)$. From here we get that $aM = 0$, which yields that $a = 0$ since 
we are assuming that $M$ is faithful as a left $A$-module. Hence $x = m + b \in M \oplus B$, which concludes the proof.

(ii) can be proved in a similar way. 
\end{proof}

\noindent {\bf Proof of Theorem \ref{key}.}
Let $\sigma$ be an automorphism of $\T$ such that $\sigma(M) = M$. Then we can find a bijective linear map $\nu_\sigma: M \to M$ such that 
\[
\sigma \left(
\begin{array}{cc}
0 & m  \\
  & 0 
\end{array}
\right) =
\left(
\begin{array}{ccc}
0 && \nu_\sigma(m)  \\
  && 0 
\end{array}
\right),
\]
for all $m \in M$. Next, keeping in mind that the left and right annihilator of an ideal invariant under $\sigma$ are also invariant under $\sigma$, Lemma \ref{lann&rann} applies to get that $\sigma(A \oplus M) \subseteq A \oplus M$ and $\sigma(M \oplus B) \subseteq M \oplus B$. In particular, $\sigma(A) \subseteq A \oplus M$ and $\sigma(B) \subseteq M \oplus B$, which imply that $\sigma(A) \cap B = 0$ and $A \cap \sigma(B) = 0$. The result now follows from \cite[Theorem 5.2.3]{Ch1}.
\qed

\medskip

\subsection{K\"{o}the's conjecture}

Recall that an ideal of an algebra $\A$ is said to be {\bf nil} if all its elements are nilpotent. The {\bf nil radical} $\Nil^*(\A)$ of $\A$, also called the {\bf K\"{o}the upper nil radical}, is defined as the sum of all nil ideals of $\A$. Note that $\Nil^*(\A)$ is the largest nil ideal of $\A$ and invariant under every automorphism of $\A$. In this subsection, we will compute the nil radical of a triangular algebra and construct a family of partible triangular algebras. 

In 1930, K\"{o}the \cite{K} conjectured that if a ring has no nonzero nil ideals, then it can not have any 
nonzero one-sided nil ideals. 
Although K\"{o}the's conjecture has been studied by a significant number of authors in the past few years, it remains an open problem. Here we show that K\"{o}the's conjecture is true for triangular algebras.

\begin{lemma} \label{nilpo}
Let $\T = \Tri(A, M, B)$ be a triangular algebra. An element $x = a + m + b$ of $\T$ is nilpotent if and only if $a$ and $b$ are nilpotent elements of $A$ and $B$, respectively.
\end{lemma}

\begin{proof}
Straightforward.
\end{proof}

\begin{proposition}
The K\"{o}the
radical of a triangular algebra $\T = \Tri(A, M, B)$ is the following triangular algebra 
$\, \Nil^*(\T) = \Tri(\Nil^*(A), \, M, \,\Nil^*(B))$.  
\end{proposition}

\begin{proof}
It is easy to see that $I:= \Tri(\Nil^*(A), \, M, \,\Nil^*(B))$ is an ideal of $\T$. Moreover, $I$ is a nil ideal of $\T$ by Lemma \ref{nilpo}. We claim that $I$ is the largest nil ideal of $\T$. In fact, let $J$ be a nil ideal of $\T$. Lemma \ref{nilpo} applies to show  
that $\pi_A(J)$ and $\pi_B(J)$ are nil ideals of $A$ and $B$, respectively; this  
implies that $\pi_A(J) \subseteq \Nil^*(A)$ and $\pi_B(J) \subseteq \Nil^*(B)$. Therefore $J \subseteq I$, as desired.
\end{proof}

\begin{corollary} \label{Kothe}
Let $\T = \Tri(A, M, B)$ be a triangular algebra. Assume that $\Nil^*(A) = \Nil^*(B) = 0$. Then the following conditions hold. 
\begin{enumerate}
\item[{\rm (i)}] $\Nil^*(\T) = M$.
\item[{\rm (ii)}] Every nilpotent element of $\T$ belongs to $M$.
\item[{\rm (iii)}] Every left (or right) nil ideal of $\T$ is contained in $M$. 
\end{enumerate} 
\end{corollary}

\begin{corollary}
K\"{o}the's conjecture holds for triangular algebras.
\end{corollary}

\begin{lemma} \label{Kothenondeg}
Any unital algebra with zero K\"{o}the radical is non-degenerate.
\end{lemma}

\begin{proof}
Straightforward. 
\end{proof}

\begin{theorem} \label{aut2}
Let $\T = \Tri(A, M, B)$ be a triangular algebra such that either $\Nil^*(A) = 0$ or $\Nil^*(B) = 0$. Then $\T$ is partible.
\end{theorem}

\begin{proof}
It follows from Corollary \ref{nondegpartible} by an application of Lemma \ref{Kothenondeg}.
\end{proof}

\begin{proof}
Let $\sigma$ be an automorphism of $\T$. By Corollary \ref{Kothe} (i) we have that $\Nil^*(\T) = M$. In particular, we have that $\sigma(M) = M$. The result now follows from Theorem \ref{key}.
\end{proof}

\begin{remark}
Notice that for example the algebra of matrices $\M_n(R)$ satisfies that $\Nil^*(\M_n(R)) = 0$, and $\M_n(R)$ has many nontrivial idempotents. 
\end{remark}

\section{Bimodule preserving automorphisms}

Let $\T = \Tri(A, M, B)$ be a triangular algebra. Along this section, $M$ needs not to be either faithful as a left $A$-module or as a right $B$-module. Inspired by Theorem \ref{key}, we introduce the following definition: 

\begin{definition}
Let $\T = \Tri(A, M, B)$ be a triangular algebra. We say that an endomorphism $\sigma$ of $\T$ is {\bf $M$-preserving} or {\bf bimodule-preserving} if $\sigma(M) = M$. 
\end{definition}

Our goal in this section is to provide a precise description of bimodule-preserving automorphisms of triangular algebras. 
Our results contribute to the development of the study of automorphism of triangular algebras initiated by Cheung \cite[Chapter 5]{Ch1}.
We first provide a class of triangular algebras such that all its automorphisms are bimodule-preserving.  

\begin{proposition}
Let $\T = \Tri(A, M, B)$ be a triangular algebra such that the algebras $A$ and $B$ satisfy {\rm Condition (I)}. Then every automorphism of $T$ is $M$-preserving. 
\end{proposition}

\begin{proof}
Let $\sigma$ be an automorphism of $\T$ and denote $e = \sigma(p)$ and $f = \sigma(q) = 1 - e$. Then applying $\sigma$ to the identities $pTq = M$ and $qTp = 0$ gives $eTf = \sigma(M)$ and $fTe = 0$, respectively. If $e = \left(\begin{array}{cc}a & m \\0 & b\end{array}\right)$ then $a$ and $b$ are idempotents of $A$ and $B$, respectively. On the other hand, from $0 = fTe = (1 - e)Te$ we get that $(1 - a)Aa = 0$ and $(1 - b)Bb = 0$. Applying that $A$ and $B$ satisfy {\rm Condition (I)} we have that $aA(1 - a) = 0$ and $bB(1 - b) = 0$. But then $\sigma(M) = eTf \subseteq M$ and $\sigma$ is $M$-preserving, as desired. 
\end{proof}

\begin{notation}
{\rm
Let $\T = \Tri(A, M, B)$ be a triangular algebra and $\sigma$ an $M$-preserving endomorphism of $\T$. We write $\sigma$ as follows: 
\[
\sigma \left(\begin{array}{cc}a & m \\0 & b\end{array}\right) = 
\left(\begin{array}{cc} 
\chi_1(a) + \gamma_3(b) & \chi_2(a) + h(m) + \gamma_2(b) \\
&
\\
0 & \chi_3(a) + \gamma_1(b) \end{array}\right),
\]
where $\chi_1: A \to A$, $\chi_2: A \to M$, $\chi_3: A \to B$, $\gamma_1: B \to B$, $\gamma_2: B \to M$, $\gamma_3: B \to A$ are linear maps and 
$h: M \to M$ is an endomorphism of the bimodule $M$.
}
\end{notation}

Our goal is to determine the necessary and sufficient condition on $x_i$, $y_i$ and $h$ to $\sigma$ be $M$-preserving.

\begin{theorem} \label{MpreservingThm0}
The following conditions hold for a bimodule-preserving endomorphism $\sigma$ of a triangular algebra $\T = \Tri(A, M, B)$. 
\begin{enumerate}
\item[\rm (i)] $\chi_1: A \to A$, $\gamma_3: B \to A$ are algebra homomorphisms, and $\Imm (\chi_1)\Imm (\gamma_3) = \Imm (\gamma_3)\Imm (\chi_1) = 0$.
\item[\rm (ii)] $\Imm (\chi_1)$ and $\, \Imm(\gamma_3)$ are ideals of $A$.
\item[\rm (iii)] $A = \Imm (\chi_1) \oplus \Imm (\gamma_3)$.
\item[\rm (iv)] $\chi_3: A \to B$, $\gamma_1: B \to B$ are algebra homomorphisms, and $\Imm (\chi_3)\Imm (\gamma_1) = \Imm (\gamma_1)\Imm (\chi_3) = 0$.
\item[\rm (v)] $\Imm (\chi_3)$ and $\, \Imm(\gamma_1)$ are ideals of $B$.
\item[\rm (vi)] $B = \Imm (\chi_3) \oplus \Imm (\gamma_1)$.
\end{enumerate}
\end{theorem}  

\begin{proof}
(i). Given $x = \left(\begin{array}{cc}a & m \\0 & b\end{array}\right)$ and $x' = \left(\begin{array}{cc}a' & m' \\0 & b'\end{array}\right)$ in $\T$, from $\sigma(xx') = \sigma(x)\sigma(x')$ we get that
\begin{align*}
\chi_1(aa') + \gamma_3(bb') & = \big(\chi_1(a) + \gamma_3(b)\big)\big(\chi_1(a') + \gamma_3(b')\big) = 
\\
& = \chi_1(a)\chi_1(a') + \chi_1(a)\gamma_3(b') + \gamma_3(b)\chi_1(a') + \gamma_3(b)\gamma_3(b'),
\end{align*}
for all $a$, $a' \in A$ and $b$, $b' \in B$. It implies that $\chi_1$ and $\gamma_3$ are algebra homomorphisms and $\Imm (\chi_1)\Imm (\gamma_3) = \Imm (\gamma_3)\Imm (\chi_1) = 0$, proving (i).

\smallskip

(ii). Since $\sigma$ is  $A = \Imm (\chi_1) + \Imm (\gamma_3)$. Then given $a_0 \in A$ and $\chi_1(a)$, $\gamma_3(b)$ for $a \in A$ and $b \in B$. Writing $a_0 = \chi_1(a') + \gamma_3(b')$ for $a' \in A$ and $b'\in B$, we have that  
\begin{align*}
\chi_1(a)a_0 & = \chi_1(a)\chi_1(a') + \chi_1(a)\gamma_3(b') = \chi_1(aa') \in \Imm (\chi_1),
\\
\gamma_3(b) a_0 & = \gamma_3(b)\chi_1(a') + \gamma_3(b)\gamma_3(b') = \gamma_3(bb') \in \Imm (\gamma_3),
\end{align*}
by (i). Similarly, $a_0\chi_1(a) \in  \Imm (\chi_1)$ and $a_0\gamma_3(b) \in \Imm (\gamma_3)$, which concludes the proof of (ii).

\smallskip

(iii). It remains to show that $\Imm (\chi_1)\cap \Imm (\gamma_3) = 0$. From $A = \Imm (\chi_1) + \Imm (\gamma_3)$ we have that 
\begin{align*}
\big(\Imm (\chi_1)\cap \Imm (\gamma_3) \big)A & \subseteq \big(\Imm (\chi_1)\cap \Imm (\gamma_3) \big)\Imm (\chi_1) + \big(\Imm (\chi_1)\cap \Imm (\gamma_3) \big)\Imm (\gamma_3) \subseteq
\\
& \subseteq \Imm (\gamma_3)\Imm (\chi_1) + \Imm (\chi_1)\Imm (\gamma_3) = 0.
\end{align*}
Thus $\big(\Imm (\chi_1)\cap \Imm (\gamma_3) \big)A = 0$, which yields that $\Imm (\chi_1)\cap \Imm (\gamma_3) = 0$.

\smallskip

(iv), (v) and (vi) can be proved in a similar way.
\end{proof}

\begin{theorem} \label{MpreservingThm1}
Let $\T = \Tri(A, M, B)$ be a triangular algebra and $\sigma$ an $M$-preserving endomorphism of $\T$ such that its restriction to $M$ is an automorphism of $M$. Then the following conditions hold. 
\begin{enumerate}
\item[\rm (i)] $\Imm(\gamma_3) \subseteq \Lann_A(M)$ and $\Imm(\chi_3) \subseteq \Rann_B(M)$.
\item[\rm (ii)] $\ker(\chi_1) \subseteq \Lann_A(M)$ and $\ker(\gamma_1) \subseteq \Rann_B(M)$.
\item[\rm (iii)] $\chi_2(aa') = \chi_1(a)\chi_2(a')$, for all $a, \, a' \in A$.
\item[\rm (iv)] $\gamma_2(bb') = \gamma_2(b)\gamma_1(b')$, for all $b, \, b' \in B$.
\item[\rm (v)] $\ker(\chi_1) \subseteq \ker(\chi_2)$ and $\ker(\gamma_1) \subseteq \ker(\gamma_2)$.
\end{enumerate}
\end{theorem}

\begin{proof}
Let $x = \left(\begin{array}{cc}a & m \\0 & b\end{array}\right)$ and $x' = \left(\begin{array}{cc}a' & m' \\0 & b'\end{array}\right)$ in $\T$. From $\sigma(xx') = \sigma(x)\sigma(x')$ we have that
\begin{align*}
 \chi_2(a & a') + h(am' + mb') + \gamma_2(bb') = \big(\chi_1(a) + \gamma_3(b)\big)\big(\chi_2(a') + h(m') + \gamma_2(b')\big) + 
\\
& + \big(\chi_2(a) + h(m) + \gamma_2(b)\big)\big(\chi_3(a') + \gamma_1(b')\big). 
\end{align*}
It implies that 
\begin{align}
\chi_2(aa') &= \chi_1(a)\chi_2(a') + \chi_2(a)\chi_3(a'), \label{dos}
\\
\gamma_2(bb') & = \gamma_3(b)\gamma_2(b') + \gamma_2(b)\gamma_1(b'), \label{tres}
\\
h(am') & =  \chi_1(a)h(m'), \label{cuatro}
\\
h(mb') & = h(m)\gamma_1(b'), \label{cinco}
\\
0 & = \chi_1(a)\gamma_2(b') + \chi_2(a)\gamma_1(b'),
\\
0 & = \gamma_3(b)\chi_2(a') +  \gamma_2(b)\chi_3(a'),
\\
0 & = h(m)\chi_3(a'), \label{ocho}
\\
0 & = \gamma_3(b)h(m'). \label{nueve}
\end{align}

(i) follows from (\ref{ocho}) and (\ref{nueve}), since $h$ is an automorphism of $M$.

\smallskip  

(ii). The equality $\ker(\chi_1) \subseteq \Lann_A(M)$ can be obtained from (\ref{cuatro}): 
$$0 = a h(m') = h(am'), \quad \forall \, m' \in M,$$ since $h$ is an automorphism of $M$. An application of (\ref{cinco}) proves that $\ker(\gamma_1) \subseteq \Rann_B(M)$. 

\smallskip 

(iii) (respectively, (iv)) follows from (\ref{dos}) (respectively, (\ref{tres})) by an application of $\Imm(\chi_3) \subseteq \Rann_B(M)$ (respectively,  $\Imm(\gamma_3) \subseteq \Lann_A(M)$).

\smallskip 

(v). Making $a' = 1$ in (\ref{tres}) gives that $\chi_2(a) = \chi_1(a)\chi_2(1)$, for all $a \in A$, which implies that  
$\ker(\chi_1) \subseteq \ker(\chi_2)$. Similarly, $\ker(\gamma_1) \subseteq \ker(\gamma_2)$. 
\end{proof}

\begin{remark}
{\rm 
Notice that condition (iii) in theorem above is equivalent to the following one:

(iii)' $\, \, \chi_2(a) = \chi_1(a)\chi_2(1), \, \, \mbox{for all} \, \, a \in A.$

\smallskip 

\noindent It is clear that (iii) implies (iii)' by making $a' = 1$. Conversely, assume that (iii)' holds and take $a, \, a' \in A$, from $\chi_2(aa') = \chi_1(aa')\chi_2(1)$ applying that $\chi_1$ is an algebra homomorphism we have that 
$$\chi_2(aa') = \chi_1(aa')\chi_2(1) = \chi_1(a)\chi_1(a')\chi_2(1) \stackrel{\rm (iii)'}{=}\chi_1(a)\chi_2(a').$$
In a similar way, one can show that condition (iv) is equivalent to the following one:

(iv)' $\, \,\gamma_2(b) = \gamma_2(1)\gamma_1(b), \, \, \mbox{for all} \, \, b \in B.$
}
\end{remark}

\begin{theorem}
Let $\T = \Tri(A, M, B)$ be a triangular algebra and $\sigma$ an $M$-preserving endomorphism of $\T$. Then $\sigma$ is a monomorphism of $\T$ if and only if the following three conditions hold.
\begin{enumerate}
\item[\rm (M1)] $\ker (h) = 0$. 
\item[\rm (M2)] $\ker(\chi_1) \cap \ker(\chi_3) = 0$.
\item[\rm (M3)] $\ker(\gamma_1) \cap \ker(\gamma_3) = 0$.
\end{enumerate} 
\end{theorem}

\begin{proof}
Assume first that $\sigma$ is an $M$-preserving monomorphism of $\T$. Then (M1) clearly follows. If $a \in \ker(\chi_1) \cap \ker(\chi_3)$ then $a \in \ker(\chi_2)$ by Theorem \ref{MpreservingThm1} (v). But also
$$\sigma \left(\begin{array}{cc}a & 0  \\0 & 0 \end{array}\right)
= \left(\begin{array}{cc}\chi_1(a) & \chi_2(a)  \\0 & \chi_3(a) \end{array}\right) = 
\left(\begin{array}{cc} 0 & 0  \\0 & 0 \end{array}\right),$$
which implies that $a = 0$. It shows (M2). Similarly, (M3) can be proved.  

Conversely, assume that (M1), (M2) and (M3) hold. If $x = \left(\begin{array}{cc}a & m  \\0 & b \end{array}\right) \in \ker(\sigma)$ then $px = \left(\begin{array}{cc}a & m  \\0 & 0 \end{array}\right) \in \ker(\sigma)$. But 
$\sigma(px) = \left(\begin{array}{cc}\chi_1(a) & \chi_2(a) + h(m) \\0 & \chi_3(a) \end{array}\right)$. Thus $a \in \ker(\chi_1) \cap \ker(\chi_3) \stackrel{\rm (M2)}{=} 0$ and so $a = 0$, which yields that $h(m) = 0$, and by (M1) we get that $m = 0$. Apply that $x \in \ker (\sigma)$ to get that $b \in \ker(\gamma_1) \cap \ker(\gamma_3) \stackrel{\rm (M3)}{=} 0$, which finishes the proof. 
\end{proof}

\begin{theorem}
Let $\T = \Tri(A, M, B)$ be a triangular algebra and $\sigma$ an $M$-preserving endomorphism of $\T$. Then $\sigma$ is an epimorphism of $\T$ if and only if the following two conditions hold.
\begin{enumerate}
\item[\rm (E1)] $h$ is an epimorphism of $M$.
\item[\rm (E2)] $A = \chi_1\big(\ker (\chi_3)\big) \cap \gamma_3 \big(\ker(\gamma_1)\big)$. 
\item[\rm (E3)] $B = \gamma_1 \big(\ker(\gamma_3)\big) \cap \chi_3\big(\ker (\chi_1)\big)$.
\end{enumerate} 
\end{theorem}

\begin{proof}
Assume first that $\sigma$ is an $M$-preserving epimorphism of $\T$. It is clear that (E1) follows. Given $\left(\begin{array}{cc}a_0 & 0  \\0 & 0 \end{array}\right) \in \T$ we find an element $x = \left(\begin{array}{cc}a & m  \\0 & b \end{array}\right) \in \T$ such that 
\[
\sigma(x) = \left(\begin{array}{cc} 
\chi_1(a) + \gamma_3(b) & \chi_2(a) + h(m) + \gamma_2(b) \\
&
\\
0 & \chi_3(a) + \gamma_1(b) \end{array}\right) = 
\left(\begin{array}{cc}a_0 & 0  \\0 & 0 \end{array}\right),
\]
which implies that 
\begin{align*}
\chi_1(a) + \gamma_3(b) & = a_0,
\\
\chi_2(a) + h(m) + \gamma_2(b)  & = 0, 
\\
\chi_3(a) + \gamma_1(b)  & = 0.
\end{align*}
Theorem \label{MpreservingThm0} (vi) allows us to conclude that $a \in \ker(\chi_3)$ and $b \in \ker(\gamma_1)$. It shows (E2). Similarly, (E3) follows. 

Conversely, assume that (E1), (E2) and (E3) hold. Then given $a_0 \in A$, apply (E2) to write it as $a_0 = \chi_1(a) + \gamma_3(b)$, where $\chi_3(a) = 0$ and $\gamma_1(b) = 0$. Thus:
\[
\left(\begin{array}{cc}a_0 & 0  \\0 & 0 \end{array}\right) = \sigma \left(\begin{array}{cc}a & m  \\0 & b \end{array}\right)
\in \Imm (\sigma), 
\] 
where $m\in M$ is such that $h(m) = -\chi_2(a) - \gamma_2(b)$. It proves that $\left(\begin{array}{cc}A & 0  \\0 & 0 \end{array}\right) \subseteq \Imm (\sigma)$. Similarly, $\left(\begin{array}{cc}0 & 0  \\0 & B \end{array}\right) \subseteq \Imm (\sigma)$. To finish, apply (E1). 
\end{proof}

\begin{theorem}
Let $\T = \Tri(A, M, B)$ be a triangular algebra. An endomorphism $\sigma:T \to T$ is an $M$-preserving automorphism of $\T$ if and only if the associated maps $\chi_i$, $\gamma_j$ and $h$ satisfy the following conditions:
\begin{enumerate}
\item[\rm (i)] $A = \Imm (\chi_1) \oplus \Imm (\gamma_3)$, $B = \Imm (\chi_3) \oplus \Imm (\gamma_1)$.
\item[\rm (ii)] $h$ is an automorphism of $M$.
\item[\rm (iii)] $\chi_1$, $\gamma_1$, $\chi_3$ and $\gamma_3$ are algebra homomorphisms, and 
\begin{align*}
A & = \Imm (\chi_1) \oplus \Imm (\gamma_3) = \chi_1\big(\ker (\chi_3)\big) \cap \gamma_3 \big(\ker(\gamma_1)\big),
\\
B & = \Imm (\chi_3) \oplus \Imm (\gamma_1) = \gamma_1 \big(\ker(\gamma_3)\big) \cap \chi_3\big(\ker (\chi_1)\big),
\end{align*}
where $\Imm(\chi_1)$, $\Imm(\gamma_3)$ are ideals of $A$, and $\Imm(\gamma_1)$, $\Imm(\chi_3)$ are ideals of $B$.
\item[\rm (iv)] $\Imm(\gamma_3) \subseteq \Lann_A(M)$ and $\Imm(\chi_3) \subseteq \Rann_B(M)$.
\item[\rm (v)] $\ker(\chi_1) \subseteq \Lann_A(M)$ and $\ker(\gamma_1) \subseteq \Rann_B(M)$.
\item[\rm (vi)] $\chi_2(a) = \chi_1(a)\chi_2(1), \, \, \mbox{for all} \, \, a \in A.$
\item[\rm (vii)] $\gamma_2(b) = \gamma_2(1)\gamma_1(b), \, \, \mbox{for all} \, \, b \in B.$
\end{enumerate}
\end{theorem}

\begin{remark}
{\rm
Let $\T = \Tri(A, M, B)$ be a triangular algebra and $\sigma$ an automorphism of $\T$. There are two extreme situations:
\begin{enumerate}
\item[\small Case 1.] $\ker(\chi_1) =  \ker(\gamma_1) = 0$.
\item[\small Case 2.] $\ker(\chi_3) =  \ker(\gamma_3) = 0$.
\end{enumerate}
In Case 1, we have that $A = \ker(\chi_3)$ and $B =  \ker(\gamma_3)$, and so $\chi_3 = \gamma_3 = 0$ and $\sigma$ is indeed partible. In Case 2, we get that $A = \ker(\chi_1)$ and $B =  \ker(\gamma_1)$. Then $\chi_1 = \gamma_1 = 0$, which imply that $\chi_2 = \gamma_2 = 0$. Then (\ref{cuatro}) and (\ref{cinco}) (see the proof of Theorem \ref{MpreservingThm1}) 
} yield that $h = 0$. Notice that this situation can only happens provided $M = 0$. In such a case, we will refer to $\T$ as a {\bf diagonal triangular algebra}, while $\sigma$ will be called an {\bf ``anti-partible'' automorphism}. 
\end{remark}

\begin{theorem}
Let $\T = \Tri(A, M, B)$ be a triangular algebra and $\sigma$ an $M$-preserving automorphism of $\T$.
Then there exist ideals $I$ and $J$ of $\T$ such that
\begin{enumerate}
\item[\rm (i)] $\sigma(I) = I$ and $\sigma(J) = J$. 
\item[\rm (ii)] $\T = I \oplus J$.
\item[\rm (iii)] $\sigma|_I$ is a partible automorphism of the triangular algebra $I$.
\item[\rm (iv)] $\sigma|_J$ is an anti-partible automorphism of the diagonal triangular algebra $J$.
\end{enumerate}
\end{theorem}

\begin{proof}
Take $I = \Tri(\ker(\chi_3), M, \ker(\gamma_3))$ and $J = \Tri(\ker(\chi_1), 0, \ker(\gamma_1))$.
\end{proof}

\begin{corollary}
Let $\T = \Tri(A, M, B)$ be a triangular algebra such that the algebras $A$ and $B$ are indecomposable. Then any $M$-preserving automorphism $\sigma$ of $\T$ is either partible or anti-partible. Moreover, $\sigma$ is anti-partible if and only if $\T$ is diagonal. 
\end{corollary}

\section{$\sigma$-biderivations of partible triangular algebras}

This section is devoted to the study of $\sigma$-biderivations of partible triangular algebras. We start by proving a preliminary result. 

\begin{lemma}
Let $\T = \Tri(A, M, B)$ be a triangular algebra and $\sigma$ a partible automorphism of $\T$. Let $\sigma = \phi_z \bar{\sigma}$, where $\phi_z (x) = z^{-1}xz$, for all $x \in \T$; and $\bar \sigma$ is an automorphism of $\T$ such that $\bar \sigma(A) = A$, $\bar \sigma(M) = M$, $\bar \sigma(B) = B$. Then
\begin{enumerate}
\item[{\rm (i)}] a linear map $d: \T \to \T$ is a $\sigma$-derivation of $\T$ if and only if $\bar d(x) = z d(x)$, for all $x \in \T$, is a $\bar \sigma$-derivation. 
\item[{\rm (ii)}] a bilinear map $D: \T \times \T \to \T$ is a $\sigma$-biderivation of $\T$ if and only if $\bar D (x, y) = z D(x, y)$, for all $x, \, y \in \T$, is a $\bar \sigma$-biderivation.
\end{enumerate} 
\end{lemma}

\begin{proof}
(i) Take $x,\, y, \, z \in \T$, and let us assume first that $d$ is a $\sigma$-derivation. Then we have that 
\begin{align*}
\bar d (xy) & = z d(xy) = z\left(d(x)y + \sigma(x)d(y)\right) = \bar d(x)y + zz^{-1}\bar \sigma(x) z d(y) = 
\\
& = 
\bar d(x)y + \bar \sigma(x)\bar d(y),
\end{align*}
which shows that $\bar d$ is a $\bar \sigma$-derivation. Conversely, suppose that $\bar d$ is a $\bar \sigma$-derivation. Then:
\begin{align*}
d (xy) & = z^{-1} \bar d(xy) = z^{-1}\left(\bar d(x)y + \bar \sigma(x) \bar d(y)\right) = z^{-1}\bar d(x)y + z^{-1}\bar \sigma(x) z z^{-1} \bar d(y) = 
\\
& = 
d(x)y + \sigma(x) d(y),
\end{align*}
as desired. 

(ii) follows from (i). 
\end{proof}

Since our objects of study are $\sigma$-biderivations of partible triangular algebras, the result above allows us to deal with a nicer class of automorphisms. More precisely, we can assume that $\sigma$ satisfies that $\sigma(A) = A$, $\sigma(M) = M$, $\sigma(B) = B$, and therefore it can be written as in \eqref{autom}. In what follows, 
we will make this assumption without further mention. 

\smallskip

The following notion will play an important role for our purposes.

\begin{definition} \cite[Definition 2.3]{YaZh}
Let $\A$ be an algebra and $\sigma$ an automorphism of $\A$. The {\bf $\sigma$-center} of $\A$ is the set $\Z_\sigma (\A)$ given by 
\[
\Z_\sigma (\A) = \{ \lambda \in \A : \sigma(x)\lambda = \lambda x, \, \, \mbox{ for all } \, x \in \A  \}.
\]
A linear map from $\A$ to $\Z_\sigma (\A)$ will be called {\bf $\sigma$-central}.
\end{definition}
 
\begin{lemma} \cite[Lemma 2.5]{YaZh}
Let $\A$ be an algebra and $\sigma$ an automorphism of $\A$. The $\sigma$-center of $\A$ is a subspace of $\A$
which is 
invariant under $\sigma$; in other words: $\sigma(\lambda) \in \Z_\sigma (\A)$ for every $\lambda \in  \Z_\sigma (\A)$.
\end{lemma}

\begin{lemma} \label{lemmacenter}
Let $\T = \Tri(A, M, B)$ be a triangular algebra and $\sigma$ an automorphism of $\T$ such that $\sigma(A) = A$, $\sigma(M) = M$, $\sigma(B) = B$, written as in \eqref{autom}. Then:
\begin{align*}
\Z_\sigma (\T) =
\left\{
\left(
\begin{array}{cc}
a & 0  \\
  &       b 
\end{array}
\right) \in \T: \, a \in \Z_{f_\sigma} (A), \, b \in \Z_{g_\sigma} (B), \, am = \nu_ \sigma(m)b, 
\, \forall \, m \in M 
\right\}.
\end{align*}
Moreover, there exists a unique algebra isomorphism $\eta: \pi_B (\Z_ \sigma (\T)) \to \pi_A (\Z_\sigma(\T))$ such that $\eta(b)m = \nu_\sigma(m)b$, for all $b\in \Z_\sigma (\T)$, $m \in M$. 
\end{lemma}

\begin{proof}
Straightforward. 
\end{proof}

Let $\A$ be an algebra and $\sigma$ an automorphism of $\A$. We introduce a new bilinear operation on $\A$:
\begin{equation} \label{newcomm}
[x, y]_\sigma = \sigma(x)y - yx, 
\, \, \mbox{ for all } \, \, x, y \in \A.
\end{equation}
We will call \eqref{newcomm} the {\bf $\sigma$-commutator} of $\A$. Note that
\[
\Z_\sigma (\A) = \{ \lambda \in \A : [x, \lambda]_\sigma = 0 \, \, \mbox{ for all } \, x \in \A  \}.
\]
The next result collects some elementary properties of the $\sigma$-commutator. We omit its proof since it consists of elementary calculations.

\begin{lemma} \label{propertiesnewcomm}
Let $\A$ be an algebra and $\sigma$ an automorphism of $\A$. For $x, y, z \in \A$ it follows that 
\begin{itemize}
\item[{\rm (i)}] $[xy, z]_\sigma = [x, z]_\sigma \, y + \sigma(x)[y, z]_\sigma$,
\item[{\rm (ii)}] $[x, [y, z]_\sigma]_\sigma = [[x, y], z]_\sigma + [y, [x, z]_\sigma]_\sigma$.
\end{itemize}
\end{lemma}

Set $x_0 \in \A$. Notice that Lemma \ref{propertiesnewcomm} (i) says that the map $\delta_{x_0}: \A \to \A$ given by 
\begin{equation} \label{sigmainnerder}
\delta_{x_0} (x) = [x, x_0]_\sigma, \quad 
\forall \, x \in \A
\end{equation}
is a $\sigma$-derivation of $\A$. We will refer to \eqref{sigmainnerder} as an {\bf inner $\sigma$-derivation}.

We continue by investigating what should be the suitable generalization of an inner biderivation in the $\sigma$-maps setting. 

\begin{proposition}
Let $\A$ be a noncommutative algebra, $\sigma$ an automorphism of $\A$ and $\lambda \in \Z_\sigma(\A)$. Then the map $\Delta_\lambda: \A \times \A \to \A$ given by
\begin{equation} \label{innbi}
\Delta_\lambda(x, y) = \lambda[x, y], \quad 
\forall \, x, y \in \A
\end{equation}
is a $\sigma$-biderivation of $\A$.
\end{proposition}

\begin{proof}
Take $x, y, zÊ\in \A$. 
\begin{align*}
\Delta_\lambda(x, z)y + \sigma(x) \Delta_\lambda(y, z) & = \lambda xzy - \lambda zxy + \sigma(x) \lambda yz - \sigma(x) \lambda zy = 
\\
& = \lambda xzy - \lambda zxy + \lambda xyz- \lambda xzy = \lambda xyz - \lambda zxy,
\end{align*}
since
$\lambda \in \Z_\sigma(\A)$. On the other hand, we have that 
\[
\Delta_\lambda(xy, z) = \lambda xyz - \lambda zxy.
\]
Therefore, $\Delta_\lambda(xy, z) = \Delta_\lambda(x, z)y + \sigma(x) \Delta_\lambda(y, z)$. Similarly, one can show that $\Delta_\lambda(x, yz) = \Delta_\lambda(x, y)z + \sigma(y) \Delta_\lambda(x, z)$, concluding the proof.
\end{proof}

The result above makes it possible to introduce the following concept.

\begin{definition} \label{sigmainn}
Let $\A$ be a noncommutative algebra and $\sigma$ an automorphism of $\A$. Maps of the form \eqref{innbi} will be called {\bf inner $\sigma$-biderivations}.
\end{definition}

One of our main results in this section is the following theorem. Of course, any inner  $\sigma$-biderivation satisfies 
$\Delta_{\lambda}(x,x) =0$, for every $x \in \A$. For a triangular algebra and the idempotent $p$ defined in Notation (\ref{notation}), the theorem provides sufficient conditions to ensure that the $\sigma$-biderivations vanishing at $(p, p)$ are indeed inner. At this point, a natural question arises:

\begin{quest} \label{question1}
{\rm What can be said about the $\sigma$-biderivations of $\T$ which do not vanish at $(p, p)$?}
\end{quest}

This question will be treated in Theorem \ref{ext0} below.

\begin{theorem} \label{innercond}
Let $\T = \Tri(A, M, B)$ be a partible triangular algebra and $\sigma$ an automorphism of $\T$. Suppose that $\T$ satisfies the following conditions:
\begin{enumerate}
\item[{\rm (i)}] $\pi_A (\Z_\sigma (\T)) = \Z_{f_\sigma} (A)$ and $\, \pi_B (\Z_\sigma (\T)) = \Z_{g_\sigma} (B)$.
\item[{\rm (ii)}] Either $A$ or $B$ is noncommutative.
\item[{\rm (iii)}] If $\lambda \in \Z_\sigma(\T)$ satisfies $\lambda x = 0$, for some nonzero element $x$ in $\T$, then $\lambda = 0$.
\item[{\rm (iv)}] Every linear map $\xi:M \to M$ satisfying $\xi(amb) = f_\sigma(a)\xi(m)b$, for all $a\in A$, $m\in M$, $b \in B$, can be expressed as $\xi(m) = \lambda_0 m + \nu_\sigma(m) \mu_0$, for all $m\in M$ and certain $\lambda_0 \in \Z_{f_\sigma} (A)$ and $\mu_0 \in \Z_{g_\sigma} (B)$.
\end{enumerate}
Then every $\sigma$-biderivation $D$ of $\T$ such that $D(p, p) = 0$ is an inner $\sigma$-biderivation.
\end{theorem}

\begin{remark}
The results of the present section generalize the results of \cite[Section 4]{Bk2} (see Section \ref{comments} for details). Theorem \ref{innercond} is the analogue of \cite[Theorem 4.11]{Bk2} for $\sigma$-maps. (See also \cite[Remark 4.12]{Bk2}.) Note that Condition (iv) could be replaced by the following condition:

(iv)$'$ Every $\sigma$-derivation of $\T$ is inner.  

\noindent The $\sigma$-derivations of triangular algebras have been already studied by Han and Wei \cite{HW}; 
however, the innerness of $\sigma$-derivations has not been treated yet. It is not difficult to prove that a $\sigma$-derivation of $\T$ is inner if and only if its associated linear map $\xi: M\to M$ 
(see \cite[Theorem 3.12]{HW} for a precise description of $\sigma$-derivations) is of the following form:
\[
\xi(m) = \lambda_0 m + \nu_\sigma(m) \mu_0, \quad \forall \, m \in M, 
\] 
for certain $\lambda_0 \in \Z_{f_\sigma} (A)$ and $\mu_0 \in \Z_{g_\sigma} (B)$. From 
this, one can easily derive that Condition (iv)$'$ above implies Condition (iv) in Theorem \ref{innercond}.
\end{remark}

In what follows, we will collect the appropriate results and develop the machinery needed to prove Theorem \ref{innercond}.

\begin{lemma} {\rm (See \cite[Lemma 2.3]{Br2} and \cite[Lemma 4.2]{Bk2})} \label{aux}
Let $\A$ be an algebra, $\sigma$ an automorphism of $\A$,
and $D$ a $\sigma$-biderivation of $\A$. Then 
\begin{enumerate}
\item[{\rm (i)}] $D(x, y)[u, v] =Ê\sigma([x, y])D(u, v) = [\sigma(x), \sigma(y)]D(u, v)$, for all $x, y, u, v \in \A$.
\item[{\rm (ii)}] $D(x, 1) = D(1, x) = D(x, 0) = D(0, x) = 0$, for all $x \in \A$.
\item[{\rm (iii)}] For an idempotent $e$ of $\A$, it follows that 

$D(e, e) = - D(e, 1 -e) = -D(1 - e, e) = D(1 - e, 1 - e)$.
\end{enumerate}
\end{lemma}

\begin{lemma} \label{aux2}
Let $\T = \Tri(A, M, B)$ be a partible triangular algebra, $\sigma$ an automorphism of $\T$, and $D$ a $\sigma$-biderivation of $\T$. If $x, y \in \T$ are such that $[x, y]Ê= 0$, then $D(x, y) = pD(x, y)q$, where $p$, $q$ are as in Notation \ref{notation}.
\end{lemma}

\begin{proof}
Take $x, y \in \T$ with $[x, y]Ê= 0$. Considering the Peirce decomposition of $\T$ associated to the idempotent $p$ of $\T$, we can write $D(x,y)$ as follows:
\begin{equation} \label{exp}
D(x, y) = pD(x, y)p + pD(x, y)q + qD(x, y)q.
\end{equation}
For $m\in M$, applying Lemma \ref{aux} (i) we get that
\[
0 = D(p, pmq)[x, y] = \sigma([p, pmq]) D(x, y) = \sigma(pmq) D(x, y) = p\nu_\sigma(m)qD(x, y),
\]
which implies that $pMqD(x, y) = 0$, since $\nu_\sigma$ is a bijection. From this, 
we can conclude that $MqD(x, y)q = 0$, which yields $qD(x, y)q = 0$, since 
$M$ is a faithful right $B$-module. Given $m\in M$, a second application of Lemma \ref{aux} (i) produces:
\[
D(x, y)pmq = D(x, y)[p, pmq] = \sigma([x, y])D(p, pmq) = 0.
\]
Thus, $pD(x,y)pM = 0$, which yields $pD(x,y)p = 0$. The result now follows from \eqref{exp}.
\end{proof}

Motivated by the fact that extremal biderivations played an important role in the study of biderivations of triangular algebras; our next task will be to investigate what should be called an extremal $\sigma$-biderivation. 

\begin{proposition}
Let $\A$ be an algebra and $\sigma$ an automorphism of $\A$. The symmetric bilinear map 
$\psi_{x_0}: \A \times \A \to \A$ given by 
\begin{equation} \label{extbi}
\psi_{x_0} (x, y) = [x, [y, x_0]_\sigma]_\sigma, \quad 
\forall \, x, y \in \A \end{equation}
is a $\sigma$-biderivation, provided that the element $x_0 \in \A$ satisfies $x_0 \notin \Z_\sigma (\A)$ and $[[\A, \A], x_0]_\sigma = 0$.
\end{proposition}

\begin{proof}
Let $x_0 \in \A$ such that $x_0 \notin \Z_\sigma (\A)$ and $[[\A, \A], x_0]_\sigma = 0$. 
Clearly, $\psi_{x_0}$ is a bilinear map. To show its symmetry, apply Lemma \ref{propertiesnewcomm} (ii) to get that 
\[
\psi_{x_0}(x, y) = [x, [y, x_0]_\sigma]_\sigma = [[x, y], x_0]_\sigma + [y, [x, x_0]_\sigma]_\sigma = \psi_{x_0}(y, x),
\]
as desired. In order to prove that $\psi_{x_0}$ is a biderivation, it is enough to check that it is a derivation in its first argument. Given $x, y, z \in \A$, applying Lemma \ref{propertiesnewcomm} (i) we obtain that 
\begin{align*}
\psi_{x_0}(xy, z) & = [xy, [z, x_0]_\sigma]_\sigma = [x, [z, x_0]_\sigma]_\sigma \, y + \sigma(x) [y, [z, x_0]_\sigma]_\sigma = 
\\
& = \psi_{x_0}(x, z) y + \sigma(x) \psi_{x_0}(y, z),  
\end{align*}
which concludes the proof.
\end{proof}

In view of the previous result, we introduce the following terminology:

\begin{definition} \label{defiextbi}
Let $\A$ be an algebra and $\sigma$ an automorphism of $\A$. An {\bf extremal $\sigma$-biderivation} is a bilinear map of $\A$ of the form \eqref{extbi}.
\end{definition}

Now, we are in a position to answer Question \ref{question1}.

\begin{theorem} \label{ext0}
Let $\T = \Tri(A, M, B)$ be a partible triangular algebra, $\sigma$ an automorphism of $\T$, and $D$ a $\sigma$-biderivation of $\T$. Then $D$ can be written as a sum of an extremal $\sigma$-biderivation and a biderivation vanishing at $(p, p)$. Moreover, $D = \psi_{D(p,p)} + D_0$, where 
$\psi_{D(p,p)}$ is the extremal $\sigma$-biderivation associated to the element $x_{0}= D(p,p)$, and $D_0$ is a biderivation of $\T$ satisfying $D_0(p, p) = 0$.
\end{theorem}

\begin{proof}
Let $D$ be a $\sigma$-biderivation of $\T$ such that $D(p, p) \neq 0$. Note that 
from $[p, p] = 0$, we get that $D(p, p) = pD(p, p)q$, by an application of Lemma \ref{aux2}. Lemma \ref{lemmacenter} allows us to conclude that $D(p, p) \notin \Z_\sigma (\T)$. Next, we claim that $[[x, y], D(p, p)]_\sigma = 0$, for every $x, y \in \T$. In fact, given $x, y \in \T$, we have that 
\begin{align*}
[[x, y], D(p, p)]_\sigma & = \sigma([x, y])D(p, p) - D(p, p)[x, y] 
\\
& = [\sigma(x), \sigma(y)]D(p, p) - D(p, p)[x, y] = 0,
\end{align*}
by Lemma \ref{aux} (i). Therefore, it makes sense to consider the extremal $\sigma$-biderivation $\psi_{D(p,p)}$. Let $D_0 = D - \psi_{D(p,p)}$; it remains to check that $D_0(p, p) = 0$. We will show that $\psi_{D(p,p)} (p, p) = D(p, p)$.
\begin{align*}
\psi_{D(p,p)} (p, p) & = [p, [p, D(p,p)]_\sigma]_\sigma = [p, \sigma(p)D(p, p) - D(p, p)p]_\sigma 
\\
& = \sigma(p)\sigma(p) D(p, p) - \sigma(p)D(p, p)p.
\end{align*}
Taking into account that $D(p, p) = pD(p, p)q$ and that $\sigma(p) = \left(
\begin{array}{cc}
1_A & 0  \\
    & 0 
\end{array}
\right)$, the right hand side of the equality above becomes $D(p, p)$; that is, $\psi_{D(p,p)} (p, p) = D(p, p)$, which finishes the proof.
\end{proof}

{\bf Proof of Theorem \ref{innercond}.} Let $D$ be a $\sigma$-biderivation of $\T$ such that 
$D(p, p) = 0$. Take 
\[
x = \left(
\begin{array}{cc}
a_x & m_x  \\
    & b_x 
\end{array}
\right) = a_x + m_x + b_x, \quad \quad
y = \left(
\begin{array}{cc}
a_y & m_y  \\
    & b_y 
\end{array}
\right) = a_y + m_y + b_y,
\]
two arbitrary elements of $\T$. The bilinearity of $D$ implies that
\begin{align} \label{Faux3}
D(x, y) & = D(a_x, a_y) +  D(a_x, m_y) +  D(a_x, b_y) + D(m_x, a_y) +  D(m_x, m_y) 
\\ \notag
& +  D(m_x, b_y) + D(m_x, a_y) +  D(m_x, m_y) +  D(m_x, b_y) + D(b_x, a_y) 
\\ \notag
& +  D(b_x, m_y) +  D(b_x, b_y).
\end{align}
Note that an application of Lemma \ref{aux} yields that $D(p, q) = D(q, p) = D(q, q) = 0$. In what follows, we will distinguish several cases: 

\noindent $\diamond$ {\it Case 1.} $D(a, b) = D(b, a) = 0$, 
for all $a \in A$ and $b \in B$.

Applying Lemma \ref{aux2},
taking into account that $[a, b] = 0$,
we get that 
$D(a, b) = pD(a, b)q \,$ and $D(b, a) = pD(b, a)q$. Thus:
\begin{align*}
D(a, b) & = pD(a, b)q = p D(ap, qb)q = p D(a, qb) pq + p \sigma(a) D(p, qb)q = \sigma(a) D(p, qb)q
\\
& = \sigma(a) D(p, q)bq + \sigma(a) \sigma(q) D(p, b) q = \sigma(aq) D(p, b) q = 0.
\end{align*}
It remains to check that $D(b, a) = 0$. Similar calculations give:
\begin{align*}
D(b, a) & = pD(b, a)q = p D(qb, ap)q = p D(q, ap) bq + p \sigma(q) D(b, ap)q = pD(q, ae)b
\\
& = p D(q, a) pb  + p\sigma(a) D(q, p) b  = 0,
\end{align*}
which finishes the proof of {\it Case 1}.

\smallskip 

\noindent $\diamond$ {\it Case 2.} $D(p, b) = D(b, p) = D(p, a) = D(a, p) = 0$, 
for all $a \in A$ and $b \in B$.

By Case 1 we have that $D(p, b) = D(b, p)$, 
for all $b \in B$. Given $a \in A$, 
an application of Lemma \ref{aux} (ii) gives:
\begin{align*}
0 & = D(1, a) = D(p, a) + D(q, a) \stackrel{\small {\it Case \, 1}}{=} D(p, a),
\\
0 & = D(a, 1) = D(a, p) + D(a, q) \stackrel{\small {\it Case \, 1}}{=} D(a, p),
\end{align*}
concluding the proof of {\it Case 2}.

\smallskip 

\noindent $\diamond$ {\it Case 3.} $D(q, b) = D(b, q) = D(q, a) = D(a, q) = 0$, 
for all $a \in A$ and $b \in B$.

\smallskip

The proof is analogous to that of Case \,2.

\smallskip

\noindent $\diamond$ {\it Case 4.} $D(a, m)p = D(m, a)p = 0$, 
for all $a \in A$ and $m \in M$.

Given $a \in A$ and $m\in M$, we have that 
\begin{align*}
D(a, m)p & = D(a, mq)p = \sigma(m)D(a, q)p + D(a, m)qp \stackrel{\small {\it Case \, 3}}{=}0,
\\
D(m, a)p & = D(mq, a)p = \sigma(m)D(q, a)p + D(m, a)qp \stackrel{\small {\it Case \, 3}}{=}0,
\end{align*}
finishing the proof of {\it Case 4}.

\smallskip

The next case can be proved in a similar way.

\noindent $\diamond$ {\it Case 5.} $\sigma(q) D(b, m) = \sigma(q) D(m, b) = 0$,
for all $m \in M$ and $b \in B$.

\smallskip 

\noindent $\diamond$ {\it Case 6.} There exists $\tilde{\lambda}_0 \in \Z_{f_\sigma}(A)$ such that 
\[
D(a, m) = - D(m, a) = \tilde{\lambda}_0 am, \quad \quad D(m, b) = - D(b, m) = \tilde{\lambda}_0 mb,
\]
for all $a \in A$, $m \in M$ and $b \in B$.

Consider the map $\xi:M \to M$ given by $\xi(m) = D(p, m)$, 
for all $m\in M$. Notice that $\xi$ is well-defined, i.e., $\xi(m) \in M$, 
for all $m \in M$. In fact, given $m\in M$, 
we have that 
\begin{align*}
D(p, m) & = D(p, pmq) = D(p, pm)q + \sigma(pm)D(p, q) = \sigma(p)D(p, m)q + D(p, p)mq 
\\
& = \sigma(p)D(p, m)q \in M,
\end{align*}
since $\sigma(p) = \left(
\begin{array}{cc}
1_A & 0  \\
  & 0
\end{array}
\right)$. Trivially, $\xi$ is an additive map. Moreover, $\xi$ satisfies $\xi(amb) = f_\sigma(a)\xi(m)b$, 
for all $a\in A$, $m\in M$, $b \in B$. In fact, given $a\in A$, $m\in M$, $b \in B$, 
we have that 
\begin{align*}
\xi (amb) & = D(p, amb) = \sigma(a)D(p, mb) + D(p, a)mb \stackrel{\small {\it Case \, 2}}{=} \sigma(a) \sigma(m) D(p, b) 
\\
& + \sigma(a) D(p, m) b \stackrel{\small {\it Case \, 2}}{=} \sigma(a) D(p, m) b = f_\sigma(a)\xi(m)b,
\end{align*}
since $\sigma(a) = \left(
\begin{array}{cc}
f_\sigma(a) & 0  \\
     & 0
\end{array}
\right)$, 
and $D(p, m) \in M$. Thus Condition (iv) implies that there exist $\lambda_0 \in \Z_{f_\sigma} (A)$ and $\mu_0 \in \Z_{g_\sigma} (B)$ such that $\xi(m) = \lambda_0 m + \nu_\sigma(m) \mu_0$, for all $m\in M$.
Lemma \ref{lemmacenter} and Condition (i) allow us to consider the element $\tilde{\lambda}_0 := \lambda_0  + \eta(\mu_0) \in \Z_{f_\sigma} (A)$. The calculations above jointly with a second use of Lemma \ref{lemmacenter} imply that 
\[
D(p, m) = \xi(m) =  \lambda_0 m + \nu_\sigma(m) \mu_0 = (\lambda_0  + \eta(\mu_0))m = \tilde{\lambda}_0m, \quad 
\forall \, m \in M.
\]
Similarly, one can find $\tilde{\mu}_0 \in \Z_{f_\sigma} (A)$ such that $D(m, p) = \tilde{\mu}_0m$, 
for all $m\in M$. Next, we claim that $\tilde{\lambda}_0 + \tilde{\mu}_0 = 0$. From Condition (ii) we have that either $A$ or $B$ is noncommutative. Let us assume, for example, that $A$ is noncommutative. Then, there exist $a, a' \in A$ such that $[a, a'] \neq 0$. Applying Lemma \ref{aux} (i), we get that 
\begin{align*}
D(a, a')[p, m] & = \sigma([a, a']) D(p, m) = f_\sigma([a, a']) \tilde{\lambda}_0 m = \tilde{\lambda}_0 [a, a']m,
\\
D(a, a')[m, p] & = \sigma([a, a']) D(m, p) = f_\sigma([a, a']) \tilde{\mu}_0 m = \tilde{\mu}_0 [a, a']m,
\end{align*}
for all $m \in M$. Thus: $(\tilde{\lambda}_0 + \tilde{\mu}_0)[a, a']M = 0$, which implies that $(\tilde{\lambda}_0 + \tilde{\mu}_0)[a, a'] = 0$, since $M$ is a faithful left $A$-module. From this,
 Condition (iii) applies to show  
 that $\tilde{\lambda}_0 + \tilde{\mu}_0 = 0$, which finishes the proof of our claim. Notice that we have just proved that 
\[
D(p, m) = - D(m, p) = \tilde{\lambda}_0m, \quad 
\forall \, m\in M.
\]
From here, an application of Lemma \ref{aux} (ii) gives that
\[
D(m, q) = - D(q, m) = \tilde{\lambda}_0m, \quad 
\forall \, m\in M.
\] 
For $a \in A$, it follows that 
\begin{align*}
D(a, m) & = D(ap, m) = \sigma(a) D(p, m) + D(a, m)p \stackrel{\small {\it Case \, 4}}{=} f_\sigma(a) \tilde{\lambda}_0m = \tilde{\lambda}_0 a m,
\\
D(m, a) & = D(m, ap) = \sigma(a) D(m, p) + D(m, a)p \stackrel{\small {\it Case \, 4}}{=} 
-f_\sigma(a) \tilde{\lambda}_0m = -\tilde{\lambda}_0 a m, 
\end{align*}
for all $m \in M$. Proceeding as above, applying {\it Case 5} one can prove that 
\[
D(m, b) = - D(b, m) = \tilde{\lambda}_0 mb,
\]
for all $m\in M$ and $b\in B$, as desired.

\smallskip 

\noindent $\diamond$ {\it Case 7.} $D(a, a') = \tilde{\lambda}_0 [a, a']$ for every $a, a' \in A$.

Given $a, a' \in A$, we start by showing that $D(a, a') \in A$.
\begin{align*}
D(a, a') & = D(pap, a') = \sigma(p)D(ap, a') + D(p, a')ap \stackrel{\small {\it Case \, 2}}{=}
\sigma(p)\sigma(a) D(p, a') + \sigma(p)D(a, a')p 
\\
& \stackrel{\small {\it Case \, 2}}{=}
\sigma(p)D(a, a')p \in A.
\end{align*}
For $m \in M$, apply Lemma \ref{aux} (i) to obtain that 
\begin{align*}
D(a, a')m = D(a, a')[p, m] = \sigma([a, a'])D(p, m) = f_\sigma([a, a']) \tilde{\lambda}_0 m = \tilde{\lambda}_0 [a, a']m, \quad \forall \, m\in M.
\end{align*}
Thus: $(D(a, a') - \tilde{\lambda}_0 [a, a'])M = 0$, which yields that $D(a, a') = \tilde{\lambda}_0 [a, a']$.

\smallskip

\noindent $\diamond$ {\it Case 8.} $D(b, b') = \left(
\begin{array}{ccc}
0 && 0
\\
&&
\\
  && \eta^{-1}(\tilde{\lambda}_0)[b, b']
\end{array}
\right)$, for all $b, b' \in B$. 

\smallskip 

Take $b, b' \in B$ and write 
$D(b, b') =
\left( 
\begin{array}{cc}
a'' & m''  \\
    & b''
\end{array}
\right)$. On the other hand, we have that 
\begin{align*}
D(b, b') & = D(qbq, b') = \sigma(q)D(bq, b') + D(q, b')bq = \sigma(q)\sigma(b) D(q, b') + \sigma(q)D(b, b')q 
\\
& \stackrel{\small {\it Case \, 3}}{=} \sigma(q)D(b, b')q  = 
\left( 
\begin{array}{cc}
0 & 0  \\
    & b''
\end{array}
\right),
\end{align*}
which implies that $a'' = 0$ and $m'' = 0$. Given $m\in M$, from Lemma \ref{aux} (i) we obtain that 
\[
\sigma(m)D(b, b')  = \sigma([p, m]) D(b, b') = D(p, m)[b, b'] = \tilde{\lambda}_0m [b, b'] = \nu_\sigma(m)\eta^{-1}(\tilde{\lambda}_0)[b, b'].
\]
Taking into account the fact that $\sigma(m) D(b, b') = \nu_\sigma(m)b''$, we have that 
\[
\nu(m)(b'' - \eta^{-1}(\tilde{\lambda}_0)[b, b']) = 0, \quad \forall \, m\in M.
\]
Combining the bijectivity of $\nu$ with the fact that $M$ is faithful as a $B$-module,  we obtain that 
$b'' = \eta^{-1}(\tilde{\lambda}_0)$, finishing the proof of {\it Case 8}.

\smallskip

\noindent $\diamond$ {\it Case 9.} $D(m, m') = 0$, for all $m, m' \in M$.

Given $m, m' \in M$, since $[m, m'] = 0$, Lemma \ref{aux2} can be applied to get that $D(m, m') = pD(m, m')q \in M$. Next, fix $m_0 \in M$ and define a map $\xi: M \to M$ by $\xi(m) = D(m, m_0)$,
for all $m \in M$. The above calculation yields that $\xi$ is well-defined. Moreover, $\xi$ is an additive map satisfying $\xi(amb)= f(a)\xi(m)b$, for all $a\in A$, $m\in M$ and $b\in B$. In fact, let $a\in A$, $m\in M$, and $b\in B$. Then 
\begin{align*}
\xi(amb) & = D(amb, m_0) = \sigma(a)D(mb, m_0) + D(a, m_0)mb \stackrel{\small {\it Case \, 6}}{=} \sigma(a) \sigma(m) D(b, m_0) 
\\ & + \sigma(a)D(m, m_0)b + \tilde{\lambda}_0 a m_0 m b \stackrel{\small {\it Case \, 6}}{=} - \sigma(a) \sigma(m)\tilde{\lambda}_0 m_0 b + f(a)D(m, m_0)b 
\\
& = f(a)\xi(m)b,
\end{align*}
since $\tilde{\lambda}_0 m_0 = \nu_\sigma(m_0)\eta^{-1}(\tilde{\lambda}_0)$ and $\nu_\sigma(m) \nu_\sigma(m_0) = 0$. Now apply Condition (iv) to find $\lambda'' \in \Z_{f_\sigma}(A)$ and $\mu'' \in \Z_{g_\sigma}(B)$ such that $\xi(m) = \lambda''m + \nu_\sigma(m)\mu''$, for all $m\in M$. Set $\lambda_{m_0} := \lambda'' + \eta(\mu'') \in \Z_{f_\sigma}(A)$. Notice that $\xi(m) = \lambda''m + \nu_\sigma(m)\mu''= \lambda_{m_0}m$, for all $m \in M$.
From here, proceeding as in the proof of {\it Case 6}, we get that $\lambda_{m_0} = 0$, which implies that $\xi(m) = D(m, m_0) = 0$, for all $m \in M$ and $m_0 \in M$. 

Taking into account what has already been proved, \eqref{Faux3} can be rewritten as:
\begin{align*} 
D(x, y) & = \tilde{\lambda}_0[a_x, a_y] +  \tilde{\lambda}_0 a_x m_y - \tilde{\lambda}_0 a_y m_x + \tilde{\lambda}_0 m_x b_y - \tilde{\lambda}_0 m_y b_x + \eta^{-1}(\tilde{\lambda}_0)[b_x, b_y].
\end{align*}

To finish, we will show that $D = \Delta_\lambda$, 
where $\Delta_\lambda$ is the inner $\sigma$-biderivation associated to the element 
$\lambda = 
\left(
\begin{array}{ccc}
\tilde{\lambda}_0 && 0  
\\
&&
\\
  && \eta^{-1}(\tilde{\lambda}_0)
\end{array}
\right)$. Notice that it makes sense to consider $\Delta_\lambda$ since $\lambda \in \Z_\sigma(\T)$. On the other hand, from 
\[
[x, y] = 
\left(
\begin{array}{ccc}
[a_x, a_y] && a_xm_y + m_xb_y -a_ym_x - m_y b_x  \\
&&
\\
  && [b_x, b_y]
\end{array}
\right),
\]
it is easy to check that $\tilde{\lambda}_0 [x, y]Ê= D(x, y)$,
which concludes the proof of the theorem.
\qed


\section{$\sigma$-commuting maps of triangular algebras}


Let $\A$ be an algebra and $\sigma$ an automorphism of $\A$. Notice that in terms of the $\sigma$-commutator, a linear map $\Theta$ of $\A$ is $\sigma$-commuting if it satisfies 
\begin{equation} \label{eqcomm1}
[x, \Theta(x)]_\sigma = 0,  
\quad 
\forall \, x \in \A. \tag{$\sigma$-c1}
\end{equation}
A linearization of \eqref{eqcomm1} gives the equivalent condition 
\begin{equation} \label{eqcomm2}
[x, \Theta(y)]_\sigma = - [y, \Theta(x)]_\sigma, \quad 
\forall \, x, y \in \A. \tag{$\sigma$-c2}
\end{equation}
Given $\lambda \in \Z_\sigma(A)$ and a $\sigma$-central map $\Omega$ of $\A$, i.e., a linear map $\Omega: \A \to \Z_\sigma(A)$, it is straightforward to check that the map $\Theta: \A \to \A$ defined by 
\begin{equation} \label{sigmaProp}
\Theta(x) = \lambda x + \Omega(x), \quad 
\forall \, x \in \A
\end{equation}
is $\sigma$-commuting. Motivated by this fact and keeping in mind the concept of proper commuting maps,  
we introduce the following notion:

\begin{definition}
Let $\A$ be an algebra and $\sigma$ an automorphism of $\A$. Maps of the form \eqref{sigmaProp} will be called {\bf proper $\sigma$-commuting maps}.
\end{definition}

In this section, we will investigate when a $\sigma$-commuting map of a partible triangular algebra is proper. Sufficient conditions will be given on a partible triangular algebra for all its $\sigma$-commuting maps to be proper. 

Let us start with the following elementary result:

\begin{lemma}
Let $\T = \Tri(A, M, B)$ be a triangular algebra and $\sigma$ a partible automorphism of $\T$. Let $\sigma = \phi_z \bar{\sigma}$, where $\phi_z (x) = z^{-1}xz$, for all $x \in \T$; and $\bar \sigma$ is an automorphism of $\T$ such that $\bar \sigma(A) = A$, $\bar \sigma(M) = M$, $\bar \sigma(B) = B$. Then a linear map $\Theta: \T \to \T$ is $\sigma$-commuting if and only if $\bar \Theta(x) = z \Theta(x)$, for all $x \in \T$, is $\bar \sigma$-commuting. 
\end{lemma}

Given a partible triangular algebra $\T = \Tri(A, M, B)$, the result above allows us to assume that the automorphisms $\sigma$ of $\T$ satisfying that $\sigma(A) = A$, $\sigma(M) = M$, $\sigma(B) = B$, and therefore they are written as in \eqref{autom}.

\begin{theorem} \label{descricomm}
Let $\T = \Tri(A, M, B)$ be a partible triangular algebra and $\sigma$ an automorphism of $\T$. Then every $\sigma$-commuting map $\Theta$ of $\T$ is of the following form:
\begin{equation} \tag{$\sigma$-cm} \label{descri}
\Theta 
\left( 
\begin{array}{cc}
a & m  \\
  & b
\end{array}
\right) = 
\left( 
\begin{array}{ccc}
\delta_1(a) + \delta_2(m) + \delta_3(b) && \delta_1(1_A)m - \nu_\sigma(m)\mu_1(1_A) 
  \\
  &&
  \\
  && \mu_1(a) + \mu_2(m) + \mu_3(b)
\end{array}
\right),
\end{equation}
where 
\begin{alignat*}{3}
& \delta_1: A \to A, \quad & \delta_2: M \to Z_{f_\sigma}(A), \quad  & \delta_3: B \to Z_{f_\sigma}(A), 
\\
& \mu_1: A \to Z_{g_\sigma}(B), \quad & \mu_2: M \to Z_{g_\sigma}(B), \quad  & \mu_3: B \to B,
\end{alignat*}
are linear maps such that 
\begin{itemize}
\item[{\rm (i)}] $\delta_1$ is an $f_\sigma$-commuting map of $A$,
\item[{\rm (ii)}] $\mu_3$ is a $g_\sigma$-commuting map of $B$,
\item[{\rm (iii)}] $\delta_1(a)m - \nu_\sigma(m)\mu_1(a) = f_\sigma(a)(\delta_1(1_A)m - \nu_\sigma(m)\mu_1(1_A))$,
\item[{\rm (iv)}] $\nu_\sigma(m)\mu_3(b) - \delta_3(b)m = (\nu_\sigma(m)\mu_3(1_B) - \delta_3(1_B)m )b$,
\item[{\rm (v)}] $\delta_2(m)m = \nu_\sigma(m)\mu_2(m)$,
\item[{\rm (vi)}] $\delta_1(1_A)m - \nu_\sigma(m)\mu_1(1_A) = \nu_\sigma(m)\mu_3(1_B)  - \delta_3(1_B)m$,
\end{itemize}
for all $a\in A$, $m\in M$ and $b\in B$. 
\end{theorem}

\begin{proof}
Let $\Theta$ be a $\sigma$-commuting map of $\T$; write:  
\[
\Theta 
\left( 
\begin{array}{cc}
a & m  \\
  & b
\end{array}
\right) = 
\left( 
\begin{array}{ccc}
\delta_1(a) + \delta_2(m) + \delta_3(b) && \tau_1(a) + \tau_2(m) + \tau_3(b)) 
 \\
 &&
 \\
  && \mu_1(a) + \mu_2(m) + \mu_3(b)
\end{array}
\right),
\] 
for all $\left( 
\begin{array}{cc}
a & m  \\
  & b
\end{array}
\right) \in \T$. In what follows, we will apply equations \eqref{eqcomm1} and \eqref{eqcomm2} 
many times with appropriate elements of $\T$. Apply \eqref{eqcomm1} with $x = \left( 
\begin{array}{cc}
a & 0  \\
  & 0
\end{array}
\right)$ to get that 
\begin{equation} \label{uno0}
[a, \delta_1(a)]_{f_\sigma} = 0, \quad \quad f_\sigma(a) \tau_1(a) = 0,
\end{equation}
for all $a \in A$. This shows (i) and implies that 
\begin{equation} \label{uno1}
\tau_1(1_A)= 0,
\end{equation}
by letting $a = 1_A$ in the second equation of \eqref{uno0}. A second application of \eqref{eqcomm1} with $x = \left( 
\begin{array}{cc}
0 & 0  \\
  & b
\end{array}
\right)$ proves (ii), and also that 
\begin{equation} \label{uno2}
\tau_3(1_B) = 0. 
\end{equation}
Next, take $x = \left( 
\begin{array}{cc}
a & 0  \\
  & b
\end{array}
\right)$ and $y = q$ in \eqref{eqcomm2} to obtain (using \eqref{uno2}) that 
\begin{equation} \label{two0}
[a, \delta_3(1_B)]_{f_\sigma} = 0, \quad \quad \quad [b, \mu_3(1_B)]_{g_\sigma} = 0, 
\end{equation}
\begin{equation} \label{two1}
- \tau_3 (1_B) b  - \tau_1(a)  - \tau_3(b) = 0
\end{equation} 
for all $a\in A$ and $b \in B$. By \eqref{uno2} equation \eqref{two1} becomes
\begin{equation} \label{two2}
\tau_1(a) +  \tau_3(b) = 0,
\end{equation}
for all $a\in A$ and $b \in B$. Next, 
let $a = 1_A$ (respectively, $b = 1_B$) in \eqref{two2} and apply \eqref{uno1} (respectively, \eqref{uno2}) to obtain that
\begin{equation} \label{two3}
\tau_1(a) = 0, \quad \quad \quad \tau_3(b) = 0, 
\end{equation}
for all $a \in A$ and $b\in B$. On the other hand, (v) follows from  \eqref{eqcomm1} by 
considering $x = \left( 
\begin{array}{cc}
0 & m  \\
  & 0
\end{array}
\right)$. Now letting $x = \left( 
\begin{array}{cc}
a & 0  \\
  & 0
\end{array}
\right)$ and $y = \left( 
\begin{array}{cc}
0 & m  \\
  & 0
\end{array}
\right)$ in \eqref{eqcomm2} produces the following equations:
\begin{align} 
& [a, \delta_2(m)]_{f_\sigma} = 0, \label{three0}
\\ \label{three1}
& f_\sigma(a)\tau_2(m) = \delta_1(a)m  -  \nu_\sigma(m) \mu_1(a),
\end{align}
for all $a \in A$ and $m \in M$. Note that \eqref{three0} says that $\delta_2(M) \subseteq \Z_{f_\sigma}(A)$. Let
$a = 1_A$ in \eqref{three1} to obtain that 
\begin{equation} \label{three2}
\tau_2(m)  = \delta_1(1_A)m - \nu_\sigma(m) \mu_1(1_A),
\end{equation}
for all $m\in M$. Notice that (iii) follows from \eqref{three1} and \eqref{three2}.

Now taking
$x = \left( 
\begin{array}{cc}
0 & 0  \\
  & b
\end{array}
\right)$ and $y = \left( 
\begin{array}{cc}
0 & m  \\
  & 0
\end{array}
\right)$ in \eqref{eqcomm2}, we get that 
\begin{align} 
& [b, \mu_2(m)]_{g_\sigma} = 0, \label{four0}
\\ \label{four1}
& -\tau_2(m)b  +  \nu_\sigma(m) \mu_3(b) - \delta_3(b)m = 0,
\end{align}
for all $m \in M$ and $b \in B$. Equation \eqref{four0} gives $\mu_2(M) \subseteq \Z_{g_\sigma}(B)$. An application of \eqref{four0} in \eqref{four1} allows us to write that
\begin{equation} \label{four2}
\tau_2(m) b = \nu_\sigma(m) \mu_3(b) - \delta_3(b)m,
\end{equation}
for all $m \in M$. Letting
$b = 1_B$ in \eqref{four2} we obtain that 
\begin{equation} \label{four3}
\tau_2(m)  = \nu_\sigma(m) \mu_3(1_B) - \delta_3(1_B)m,
\end{equation}
for all $m\in M$. From \eqref{four2} and \eqref{four3} we prove (iv). On the other hand, (vi) follows from \eqref{three2} and \eqref{four3}. Next notice that an application of \eqref{two3} and \eqref{three2}
gives: 
\[
\tau_1(a) + \tau_2(m) + \tau_3(b) = \delta_1(1_A)m - \nu_\sigma(m) \mu_1(1_A),
\] 
as desired. It remains to show that $\mu_1(A) \subseteq Z_{g_\sigma}(B)$ and $\delta_3(B) \subseteq Z_{f_\sigma}(A)$,
which follows from an application of \eqref{eqcomm2} with $x = \left( 
\begin{array}{cc}
a & 0  \\
  & 0
\end{array}
\right)$ and $y = \left( 
\begin{array}{cc}
0 & 0 \\
  & b
\end{array}
\right)$, finishing the proof.
\end{proof}

In what follows, we will use this theorem 
without further mention. In other words, whenever we are given a $\sigma$-commuting map $\Theta$ of a partible triangular algebra $\T$, we will assume that $\Theta$ is of the form \eqref{descri}. Our next result provides a characterization of the properness of a $\sigma$-commuting map of $\T$ in terms of its behavior with the $\sigma$-center of $\T$. 

\begin{theorem} \label{caract}
Let $\T = \Tri(A, M, B)$ be a partible triangular algebra and $\sigma$ an automorphism of $\T$. Then for every $\sigma$-commuting map $\Theta$ of $\T$, the following conditions are equivalent:
\begin{itemize}
\item[{\rm (i)}] $\Theta$ is a proper $\sigma$-commuting map.
\item[{\rm (ii)}] $\mu_1(A) \subseteq \pi_B(\Z_\sigma(\T))$, $\delta_3(B) \subseteq \pi_A(\Z_\sigma(\T))$ and 
\[
\left( 
\begin{array}{ccc}
\delta_2(m) && 0 
\\
&&
\\
  && \mu_2(m)
\end{array}
\right) \in \Z_\sigma(\T), \quad \forall \, m \in M. 
\]
\item[{\rm (iii)}] $\delta_1(1_A) \in \pi_A(\Z_\sigma(\T))$, $\mu_1(1_A) \in \pi_B(\Z_\sigma(\T))$ and 
\[
\left( 
\begin{array}{ccc}
\delta_2(m) && 0 
\\
&&
\\
  && \mu_2(m)
\end{array}
\right) \in \Z_\sigma(\T), \quad 
\forall \, m \in M.
\]
\end{itemize}
\end{theorem}

In order to prove the theorem above, we need a preliminary result.

\begin{lemma} \label{lemaux}
Let $\T = \Tri(A, M, B)$ be a triangular algebra and $\sigma$ an automorphism of $\T$ such that $\sigma(A) = A$, $\sigma(M) = M$ and $\sigma(B) = B$. Then for every $\sigma$-commuting map $\Theta$ of $\T$, it follows that 
\[
[A, A] \subseteq \mu^{-1}_1(\pi_B(\Z_\sigma(\T))) \lhd A, \quad \quad \quad 
[B, B] \subseteq \delta^{-1}_3(\pi_A(\Z_\sigma(\T))) \lhd B.
\]
\end{lemma}

\begin{proof}
Let us show, for example, that $[A, A] \subseteq \mu^{-1}_1(\pi_B(\Z_\sigma(\T))) \lhd A$. The other condition can be proved analogously.
Set $a, a' \in A$ and $m \in M$. Applications of Theorem \ref{descricomm} (iii) give the following:
\begin{align} \label{id1}
\delta_1(a'a)m - \nu_\sigma(m)\mu_1(a'a) & = f_\sigma(a'a)(\delta_1(1_A)m - \nu_\sigma(m)\mu_1(1_A)) 
\\ \notag
& = f_\sigma(a')(\delta_1(a)m - \nu_\sigma(m)\mu_1(a)),
\end{align}
\begin{align} \label{id2}
\delta_1(aa')m - \nu_\sigma(m)\mu_1(aa') & = f_\sigma(aa')(\delta_1(1_A)m - \nu_\sigma(m)\mu_1(1_A)) 
\\ \notag
& = f_\sigma(a)(f_\sigma(a')\delta_1(1_A)m - f_\sigma(a')\nu_\sigma(m)\mu_1(1_A)) 
\\ \notag
& = f_\sigma(a)(\delta_1(1_A)a'm - \nu_\sigma(a'm)\mu_1(1_A))  
\\ \notag
& = \delta_1(a)a'm - \nu_\sigma(a'm)\mu_1(a).
\end{align}
From \eqref{id1} and \eqref{id2} we get that 
\begin{align*}
\delta_1(&[a,a'])m - \nu_\sigma(m) \mu_1([a, a']) + f_\sigma(a')\delta_1(a)m - \delta_1(a)a'm + \nu_\sigma(a'm)\mu_1(a) 
\\
& - f_\sigma(a')\nu_\sigma(m)\mu_1(a) = 0.
\end{align*}
Taking into account the facts
that $\nu_\sigma(a'm) = f_\sigma(a')\nu_\sigma(m)$ and $\delta_1(A) \subseteq \Z_{f_\sigma}(A)$, the identity above becomes:
\[
\delta_1([a,a'])m = \nu_\sigma(m)\mu_1([a, a']), 
\]
which allows us to conclude that  
\[
\left(  
\begin{array}{ccc}
\delta_1([a,a']) && 0  
\\
&&
\\
  && \mu_1([a, a'])
\end{array}
\right) \in \Z_\sigma(\T),
\]
for all $a, a' \in A$. This implies that $\mu_1([A, A]) \subseteq \pi_B(\Z_\sigma(\T))$. Therefore
$[A, A] \subseteq \mu^{-1}_1(\pi_B(\Z_\sigma(\T)))$. 

Let us show now that $\mu^{-1}_1(\pi_B(\Z_\sigma(\T)))$ is an ideal of $A$. To this end, take $a \in \mu^{-1}_1(\pi_B(\Z_\sigma(\T)))$ and $a' \in A$, $m\in M$. Applying \eqref{id1}, taking into account the fact that $\nu_\sigma(m) \mu_1(a) = \eta(\mu_1(a))m$, which makes sense 
since $\mu_1(a) \in \pi_B(\Z_\sigma(\T))$, we get that 
\[
(\delta_1(a'a) - f_\sigma(a')\delta_1(a) + f_\sigma(a')\eta(\mu_1(a)))m = \nu_\sigma(m) \mu_1(a'a).
\]
This yields 
\[
\left( 
\begin{array}{ccc}
\delta_1(a'a) - f_\sigma(a')\delta_1(a) + f_\sigma(a')\eta(\mu_1(a)) && 0  
\\
&&
\\
  && \mu_1(a'a)
\end{array}
\right) \in \Z_\sigma(\T),
\]
which says that $\mu_1(a'a) \in \pi_B(\Z_\sigma(\T))$, that is, $a'a \in \mu^{-1}_1(\pi_B(\Z_\sigma(\T)))$. To show that $aa' \in \mu^{-1}_1(\pi_B(\Z_\sigma(\T)))$, 
proceed as before, this time applying \eqref{id2}. This proves that $\mu^{-1}_1(\pi_B(\Z_\sigma(\T)))$ is an ideal of $A$, as desired. 
\end{proof}

{\bf Proof of Theorem \ref{caract}.} We are going to show that (i)$\Leftrightarrow$(iii) and (ii)$\Leftrightarrow$(iii).

Let $\Theta$ be a $\sigma$-commuting map of $\T$. 

\noindent (i) $\Rightarrow$ (iii). 
Suppose that $\Theta$ is proper. Then there exist $\lambda \in \Z_\sigma(\T)$ and a linear map $\Omega: \T \to \Z_\sigma(\T)$ such that $\Theta(x) = \lambda x + \Omega(x)$, for all $x \in \T$. We can express $\lambda$ as follows:
\[
\lambda = 
\left( 
\begin{array}{ccc}
a_\lambda  && 0  
\\
&&
\\
  && \eta^{-1}(a_\lambda)  
\end{array}
\right),
\]
for some $a_\lambda \in \pi_A(\Z_\sigma(\T))$. Take $m\in M$ and compute $\Theta(x)$ for $x = \left( 
\begin{array}{cc}
0 & m \\
  & 0
\end{array}
\right)$; let us assume that $\Omega(x) = \left( 
\begin{array}{ccc}
a_m  && 0  
\\
&&
\\
  && \eta^{-1}(a_m)  
\end{array}
\right)$, where $a_m \in \pi_A(\Z_\sigma(\T))$. Then, we have that 
\begin{equation} \label{exp1}
\Theta(x) = 
\left( 
\begin{array}{ccc}
a_m  && a_\lambda m  
\\
&&
\\
  && \eta^{-1}(a_m)  
\end{array}
\right).
\end{equation}
On the other hand, we can write 
\begin{equation} \label{exp2}
\Theta(x) = 
\left( 
\begin{array}{ccc}
\delta_2(m)  && \delta_1(1_A)m - \nu_\sigma(m)\mu_1(1_A)  
\\
&&
\\
  && \mu_2(m) 
\end{array}
\right).
\end{equation}
Therefore, \eqref{exp1} and \eqref{exp2} allow us to conclude that 
\begin{align} \label{prueba0}
& a_m = \delta_2(m), \quad \quad \mu_2(m) = \eta^{-1}(a_m), \\ \label{prueba1}
& a_\lambda m  = \delta_1(1_A)m - \nu_\sigma(m)\mu_1(1_A).  
\end{align}
From \eqref{prueba0} we get that 
\[
\left( 
\begin{array}{ccc}
\delta_2(m)  && 0 
\\
&&
\\
  && \mu_2(m) 
\end{array}
\right) \in \Z_\sigma(\T),
\]
for all $mÊ\in M$. Note that \eqref{prueba1} becomes
\begin{equation} \label{prueba2}
a_\lambda m = \delta_1(1_A)m - \nu_\sigma(m)\mu_1(1_A),  \quad 
\forall \, m \in M,
\end{equation}
by an application of \eqref{prueba0}. Taking into account the fact that $a_\lambda \in \Z_\sigma(\T)$,  \eqref{prueba2} can be rewritten as
\begin{equation*}
\delta_1(1_A)m = \nu_\sigma(m) (\mu_1(1_A) + \eta^{-1}(a_\lambda)), \quad \forall \, m \in M,
\end{equation*}
which implies that
\[
\left( 
\begin{array}{ccc}
\delta_1(1_A)  && 0
\\
&&
\\
  && \mu_1(1_A) + \eta^{-1}(a_\lambda) 
\end{array}
\right) \in \Z_\sigma(\T),
\]
for all $m \in M$. In particular, we have that $\delta_1(1_A) \in \pi_A(\Z_\sigma(\T))$. On the other hand, writing \eqref{prueba2} as
\begin{equation*}
(\delta_1(1_A) - a_\lambda)m = \nu_\sigma(m)\mu_1(1_A),  \quad 
\forall \, m \in M,
\end{equation*} 
we obtain that 
\[
\left( 
\begin{array}{ccc}
\delta_1(1_A) - a_\lambda  && 0 
\\
&&
\\
  && \mu_1(1_A) 
\end{array}
\right) \in \Z_\sigma(\T),
\]
for all $m \in M$. Thus: $\mu_1(1_B) \in \pi_B(\Z_\sigma(\T))$, concluding the proof of (iii).

\smallskip 

\noindent (iii) $\Rightarrow$ (i).

Let us start by noticing that $\delta_1(1_A) \in \pi_A(\Z_\sigma(\T))$ and $\mu_1(1_A) \in \pi_B(\Z_\sigma(\T))$ allow us to write $\eta(\mu_1(1_A))$ and $\eta^{-1}(\delta_1(1_A))$, respectively. Thus, it makes sense to consider the element 
\[
\lambda:= \left( 
\begin{array}{ccc}
\delta_1(1_A) - \eta(\mu_1(1_A)) && 0 
\\
&&
\\
&& \eta^{-1}(\delta_1(1_A)) - \mu_1(1_A)
\end{array}
\right) \in \T.
\]
Note that $\lambda \in \Z_\sigma(\T)$ since 
\[
(\delta_1(1_A) - \eta(\mu_1(1_A)))m = \nu_\sigma(m)(\eta^{-1}(\delta_1(1_A)) - \mu_1(1_A)), \quad 
\forall \, m \in M.
\]
Next, we claim that $\Omega(x) := \Theta(x) - \lambda x \in \Z_\sigma(\T)$,
 for all $x\in \T$. Given $x = \left(
\begin{array}{cc}
a_x & m_x  \\
    & b_x 
\end{array}
\right) \in \T$, we have that 
\[
\lambda x = \left( 
\begin{array}{ccc}
(\delta_1(1_A) - \eta(\mu_1(1_A)))a_x && (\delta_1(1_A) - \eta(\mu_1(1_A)))m_x
\\
&&
\\
&& (\eta^{-1}(\delta_1(1_A)) - \mu_1(1_A))b_x
\end{array}
\right),
\]
which implies that
\allowdisplaybreaks
\begin{align} \label{eqproper}
\Omega(x) & = 
\left( 
\begin{array}{ccc}
\delta_1(a_x) - \delta_1(1_A)a_x + \eta(\mu_1(1_A))a_x && 0
\\
&&
\\
&&  \mu_1(a_x)
\end{array}
\right) 
\\ \notag
& +
\left( 
\begin{array}{ccc}
\delta_3(b_x) &&  - \eta^{-1}(\delta_1(1_A))b_x + \mu_1(1_A)b_x)
\\
&&
\\
&& \mu_3(b_x) - \eta^{-1}(\delta_1(1_A))b_x + \mu_1(1_A)b_x
\end{array}
\right) 
\\ \notag
& +
\left( 
\begin{array}{ccc}
\delta_2(m_x) && 0
\\
&&
\\
&& \mu_2(m_x)
\end{array}
\right). 
\end{align}
To finish, we will show that the three terms in \eqref{eqproper} are indeed in $\Z_\sigma(\T)$. Notice that the 
last term belongs to $\Z_\sigma(\T)$, by hypothesis. Given $m\in M$, it is enough to show that 
\allowdisplaybreaks
\begin{alignat*}{3}
& (\delta_1(a_x) - \delta_1(1_A)a_x + \eta(\mu_1(1_A))a_x)m  -  \nu_\sigma(m)\mu_1(a_x) & = 0,
\\
& \delta_3(b_x) m - \nu_\sigma(m)(\mu_3(b_x) - \eta^{-1}(\delta_1(1_A))b_x + \mu_1(1_A)b_x) &= 0.
\end{alignat*}
Let us start by proving the first identity above. Apply Theorem \ref{descricomm} (iii) to get that 
\begin{align*}
(\delta_1(a_x) & - \delta_1(1_A)a_x + \eta(\mu_1(1_A))a_x)m  - \nu_\sigma(m)\mu_1(a_x) =
\\
& f_\sigma(a_x)\delta_1(1_A)m - f_\sigma(a_x) \nu_\sigma(m)\mu_1(1_A) - \delta_1(1_A)a_x m + \eta(\mu_1(1_A)) a_xm = 
\\
& [a_x, \delta_1(1_A)]_{f_\sigma}m + \eta(\mu_1(1_A)) a_xm - \nu_\sigma(a_xm)\mu_1(1_A) = 0,
\end{align*}
since $\delta_1(1_A) \in \pi_A(\Z_\sigma(\T)) \subseteq \Z_{f_\sigma}(A)$ and $\mu_1(1_A) \in \pi_B(\Z_\sigma(\T))$,  
by (iii). 

It remains to prove the second of the two displayed identities above. Apply Theorem \ref{descricomm} (iv) and (vi) to obtain that  
\begin{align*}
\delta_3(b_x) m & - \nu_\sigma(m)(\mu_3(b_x) - \eta^{-1}(\delta_1(1_A))b_x + \mu_1(1_A)b_x) = 
\\
&
(\delta_3(1_B) m - \nu_\sigma(m) \mu_3(1_B))b_x + \nu_\sigma(m) \eta^{-1}(\delta_1(1_A)) b_x - \nu_\sigma(m) \mu_1(1_A)b_x = 
\\
& 
(\nu_\sigma(m) \mu_1(1_A) - \delta_1 (1_A)m) b_x + \nu_\sigma(m) \eta^{-1}(\delta_1(1_A)) b_x - \nu_\sigma(m) \mu_1(1_A)b_x = 
\\
& 
(\nu_\sigma(m) \eta^{-1}(\delta_1(1_A)) - \delta_1(1_A)m) b_x = 0,
\end{align*}
since $\delta_1(1_A) \in \pi_A(\Z_\sigma(\T))$.

\smallskip 

\noindent (ii) $\Rightarrow$ (iii). 

It is only necessary to show that $\delta_1(1_A) \in \pi_A(\Z_\sigma(\T))$. Given $m\in M$,
notice that $\delta_3(1_B) \in \pi_A(\Z_\sigma(\T))$ allows us to write $\delta_3(1)m = \nu_\sigma(m) \eta^{-1}(\delta_3(1_B))$. Keeping this fact in mind, an application of Theorem \ref{descricomm} (vi) gives the following: 
\[
\delta_1(1_A)m = \nu_\sigma(m)(\mu_1(1_A) + \mu_3(1_B) - \eta^{-1}(\delta_3(1_B))),
\]
which implies that 
\[
\left( 
\begin{array}{ccc}
\delta_1(1_A) && 0 
\\
&&
\\
&& \mu_1(1_A) + \mu_3(1_B) - \eta^{-1}(\delta_3(1_B))
\end{array}
\right) \in \Z_\sigma(\T).
\]
Thus: $\delta_1(1_A) \in \pi_A(\Z_\sigma(\T))$, proving (iii).  

\smallskip 

\noindent (iii) $\Rightarrow$ (ii). 

From $\mu_1(1_A) \in \pi_B(\Z_\sigma(\T))$, we get that $1_A \in \mu^{-1}_1(\pi_B(\Z_\sigma(\T)))$, which is an ideal of $A$ by Lemma \ref{lemaux}. Hence:  $\mu^{-1}_1(\pi_B(\Z_\sigma(\T))) = A$, and therefore $\mu_1(A) \subseteq \pi_B(\Z_\sigma(\T))$. In order to show that $\delta_3(B) \subseteq \pi_A(\Z_\sigma(\T))$, we first need to prove that $\delta_3(1_B) \in \pi_A(\Z_\sigma(\T))$. To this end, apply Theorem \ref{descricomm} (vi), 
taking into account the fact that $\delta_1(1_A)\in \pi_A(\Z_\sigma(\T))$, to obtain that 
\begin{align*}
\delta_3(1_B)m & = \nu_\sigma(m)(\mu_3(1_B) + \mu_1(1_A)) - \delta_1(1_A)m
\\
& = \nu_\sigma(m)(\mu_3(1_B) + \mu_1(1_A) - \eta^{-1}(\delta_1(1_A))), \quad  
\forall \, m \in M.
\end{align*}
This implies that 
\[
\left( 
\begin{array}{ccc}
\delta_3(1_B) && 0 
\\
&&
\\
  && \mu_1(1_A) + \mu_3(1_B) - \eta^{-1}(\delta_1(1_A))
\end{array}
\right) \in \Z_\sigma(\T),
\]
and therefore $\delta_3(1_B) \in \pi_A(\Z_\sigma(\T))$. Now take $b\in B$, $m\in M$ and apply Theorem \ref{descricomm} (iv), taking into account the fact that $\delta_3(1_B)m = \nu_\sigma(m)\eta^{-1}(\delta_3(1_B))$, to get that 
\begin{align*}
\delta_3(b)m & = \nu_\sigma(m)\mu_3(b) - \delta_3(1_B)mb + \nu_\sigma(m)\mu_3(1_B)b 
\\
& = \nu_\sigma(m)(\mu_3(b) - \eta^{-1}(\delta_3(1_B))b +  \mu_3(1_B)b),
\end{align*}
which says that 
\[
\left( 
\begin{array}{ccc}
\delta_3(b) && 0
\\
&&
\\
  &&  \mu_3(b) + \mu_3(1_B)b - \eta^{-1}(\delta_3(1_B))b
\end{array}
\right) \in \Z_\sigma(\T).
\]
Thus: $\delta_3(b) \in \pi_A(\Z_\sigma(\T)$, concluding the proof of Theorem \ref{caract}.
\qed
\bigskip

Our last result states sufficient conditions on a triangular algebra $\T$ to guarantee that all its $\sigma$-commuting maps are proper. 

\begin{theorem} \label{caractcomm}
Let $\T = \Tri(A, M, B)$ be a partible triangular algebra and $\sigma$ an automorphism of $\T$. Suppose that $\T$ satisfies the 
following conditions:
\begin{itemize}
\item[{\rm (i)}] Either $\Z_{f_\sigma}(A) = \pi_A(\Z_\sigma(\T))$ or $B = [B, B]$.
\item[{\rm (ii)}] Either $\Z_{g_\sigma}(B) = \pi_B(\Z_\sigma(\T))$ or $A = [A, A]$.
\item[{\rm (iii)}] There exists $m_0 \in M$ such that 
\[
\Z_\sigma (\T) =
\left\{
\left(
\begin{array}{cc}
a & 0  \\
  &       b 
\end{array}
\right) \in \T, \, \mbox{ such that } \, am_0 = \nu_\sigma(m_0)b 
\right\}.
\]
\end{itemize}
Then every $\sigma$-commuting map of $\T$ is proper.
\end{theorem}

\begin{proof}
Let $\Theta$ be a $\sigma$-commuting map of $\T$. By Theorem \ref{descricomm} we have that $\mu_1(A) \subseteq \Z_{g_\sigma}(B)$. If $\Z_{g_\sigma}(B) = \pi_B(\Z_\sigma(\T))$, we can conclude that $\mu_1(A) \subseteq \pi_B(\Z_\sigma(\T))$. On the other hand, if $A = [A, A]$, apply Lemma \ref{lemaux}  to get that $A \subseteq \mu^{-1}_1(\pi_B(\Z_\sigma(\T)))$, that is, $\mu_1(A) \subseteq \pi_B(\Z_\sigma(\T))$. Reasoning as above, 
now using the fact that $\delta_3(B) \subseteq \Z_{f_\sigma}(A)$ and applying (i), we get that $\delta_3(B) \subseteq \pi_A(\Z_\sigma(\T))$.

In view of Theorem \ref{caract}, to prove that $\Theta$ is proper, it remains to show that 
\begin{equation} \label{show}
\left( 
\begin{array}{ccc}
\delta_2(m) && 0 
\\
&&
\\
  && \mu_2(m)
\end{array}
\right) \in \Z_\sigma(\T), \quad 
\forall \, m \in M.
\end{equation}
From Theorem \ref{descricomm} we have that $\delta_2(M) \subseteq \Z_{f_\sigma}(A)$ and $\mu_2(M) \subseteq \Z_{g_\sigma}(B)$, which imply that 
\begin{equation} \label{center}
\delta_2(m)m = \nu_\sigma(m)\mu_2(m), \quad 
\forall \, m \in M.
\end{equation}
In particular, for $m = m_0$ we get that $\delta_2(m_0)m_0 = \nu_\sigma(m_0)\mu_2(m_0)$, which gives
\[
\left( 
\begin{array}{ccc}
\delta_2(m_0) && 0 
\\
&&
\\
  && \mu_2(m_0)
\end{array}
\right) \in \Z_\sigma(\T),
\]
by an application of (iii). Therefore:
\[
\left[ \left( 
\begin{array}{cc}
0 & m
\\
&
\\
  & 0
\end{array}
\right),
\left( 
\begin{array}{ccc}
\delta_2(m_0) && 0
\\
&&
\\
  && \mu_2(m_0)
\end{array}
\right)
\right]_\sigma = 0, \quad 
\forall \, m \in M.
\]
Expanding the above product, we get that 
\begin{equation} \label{centermix}
\delta_2(m_0)m = \nu_\sigma(m)\mu_(m_0), \quad 
\forall \, m \in M.
\end{equation}
Given $m\in M$, applying \eqref{center} with $m + m_0$ and making use of \eqref{centermix}, 
we get that 
\[
\delta_2(m)m_0 = \nu_\sigma(m_0)\mu_(m),
\]
which, by an application of (iii), proves \eqref{show}, concluding the proof of the theorem. 
\end{proof}

\begin{remark}
Let $\sigma$ be an automorphism of an algebra $\A$. Notice that, as happened with biderivations and commuting maps, the study of the properness of $\sigma$-commuting maps of $\A$ can be reduced to the study of the innerness of its $\sigma$-biderivations. Let $\Theta$ be any $\sigma$-commuting map of $\A$; applying Lemma \ref{propertiesnewcomm} (i) and \eqref{eqcomm2},
it is not difficult to prove that the map $D_\Theta: \A \times \A \to \A$ given by $D_\Theta (x, y) = [x, \Theta(y)]_\sigma$, for all $x, y \in \A$, is a $\sigma$-biderivation of $\A$. Let us assume now that $D_\Theta$ is inner, i.e., $D_\Theta (x, y) = \lambda [x, y]$, for some $\lambda \in \Z_\sigma (\A)$. This implies that the map $\Omega_\Theta(y):= \Theta(y) - \lambda y$ is $\sigma$-central; in fact:
\begin{align*}
[x, \Omega_\Theta(y)]_\sigma & = [x, \Theta(y) - \lambda y]_\sigma = \sigma(x)\Theta(y) - \sigma(x)\lambda y - \Theta(y)x + \lambda y x 
\\
& = [x, \Theta(y)]_\sigma - \lambda x y + \lambda y x = [x, \Theta(y)]_\sigma - \lambda [x, y] 
\\
& = D_\Theta(x, y) - \lambda [x, y] = 0,
\end{align*}
for all $x \in \A$. Therefore, we can conclude that $\Theta$ is a proper $\sigma$-commuting map. 

Suppose now that $\A$ is a triangular algebra and notice that $D_\Theta(p, p) = 0$. In other words, $D_\Theta$ is one of the $\sigma$-biderivations studied in Theorem \ref{innercond}. Taking into account the calculations above, under the assumptions of Theorem \ref{innercond}, we can guarantee 
that every $\sigma$-commuting map of $\A$ is proper. Nevertheless, Theorem \ref{caractcomm} above gives us 
more precise information when dealing directly with $\sigma$-commuting maps; it is a generalization of \cite[Theorem 2]{Ch2}.
\end{remark}

The next result shows that, under some mild conditions, the identity is the only commuting automorphism of a triangular algebra.

\begin{theorem}
Let $\T = \Tri(A, M, B)$ be a partible triangular algebra. If $\sigma$ is a commuting automorphism of $\T$,
then $\sigma = \Id_\T$.
\end{theorem}

\begin{proof}
Let $\sigma$ be a commuting automorphism of $\T$ and consider the elements $x = \left(
\begin{array}{cc}
a & 0  \\
  & 0 
\end{array}
\right)$ and 
$y = \left(
\begin{array}{cc}
0 & m  \\
  & 0 
\end{array}
\right)$; 
since 
$[x, \sigma(y)] + [y, \sigma(x)] = 0$, 
we get that $a\nu_\sigma(m) - f_\sigma(a)m = 0$,
for all $a \in A$ and $m \in M$. In particular, for $a = 1_A$ we get that $\nu_\sigma = \Id_M$,
which implies that 
\[
(a - f_\sigma(a))M = 0, \quad \quad \quad M(b - g_\sigma(b)) = 0,
\]
for all $a \in A$, $b \in B$. Therefore: $f_\sigma = \Id_A$ and $g_\sigma = \Id_B$, since 
$M$ is a faithful $(A, B)$-bimodule.  
\end{proof}

We close the section by showing that Posner's theorem also holds for partible triangular algebras.

\begin{theorem}
Let $\T = \Tri(A, M, B)$ be a partible triangular algebra and $\sigma$ an automorphism of $\T$. If a $\sigma$-derivation $d$ of $\T$ is $\sigma$-commuting, then $d = 0$.
\end{theorem}

\begin{proof}
Let $d$ be a $\sigma$-derivation of $\T$
which is $\sigma$-commuting. By \cite[Theorem 3.12]{HW},  
we know that $d$ is of the following form:
\begin{equation*}
d \left(
\begin{array}{cc}
a & m  \\
  & b 
\end{array}
\right) = 
\left(
\begin{array}{ccc}
d_A(a) && f_\sigma(a)m_d - m_d b  + \xi_d(m)  
\\
&&
\\
  && d_B(b) 
\end{array}
\right),
\end{equation*}
where $d_A$ is an $f_\sigma$-derivation of $A$, $d_B$ is a $g_\sigma$-derivation of $B$, $m_d$ is a fixed element of $M$, and $\xi_d: M \to M$ is a linear map
which satisfies 
\begin{equation} \label{properxi}
\xi_d(am) = d_A(a)m + f_\sigma(a)\xi(m), \quad \quad \quad \xi(mb) = \xi(m)b + \nu_\sigma(m) d_B(b),
\end{equation} 
for all $a \in A$, $b \in B$, $m \in M$. An application of \eqref{eqcomm2} with 
$x = \left(
\begin{array}{cc}
a & 0  \\
  & 0 
\end{array}
\right)$ and 
$y = \left(
\begin{array}{cc}
0 & m  \\
  & 0 
\end{array}
\right)$ gives that $f_\sigma(a)\xi(m) - d_A(a)m = 0$, 
for all $a \in A$ and $m \in M$. In particular, substituting 
$a = 1_A$ we get that $\xi = 0$. Then from \eqref{properxi} we obtain that 
\[
d_A(a)M = 0, \quad \quad \nu_\sigma(M)d_B(b) = 0,
\]
for all $a \in A$ and $b \in B$. Since 
$\nu_\sigma$ is bijective and $M$ is a faithful $(A, B)$-bimodule, we see  
that $d_A = d_B = 0$. Applying \eqref{eqcomm1} with $x = p$, 
we get that $m_d = 0$, finishing the proof.
\end{proof}

\begin{corollary}
If $d$ is a commuting derivation of a partible triangular algebra, then $d = 0$.
\end{corollary}

\section{Some facts about $(\alpha, \beta)$-biderivations, $(\alpha, \beta)$-commuting maps and generalized matrix algebras.} \label{comments}

In 2011, Xiao and Wei \cite{XW} introduced a generalization of $\sigma$-biderivations and $\sigma$-commuting maps
that they named $(\alpha, \beta)$-biderivations and $(\alpha, \beta)$-commuting maps. In their paper, they studied $(\alpha, \beta)$-biderivations and $(\alpha, \beta)$-commuting maps of nest algebras. In the present paper, we have investigated $\sigma$-biderivations and $\sigma$-commuting maps since, as will be shown below, the study of $(\alpha, \beta)$-biderivations (respectively, $(\alpha, \beta)$-commuting maps) of any algebra can be reduced to the study of its $\sigma$-biderivations (respectively, $\sigma$-commuting maps).

Let $\A$ be an algebra and $\alpha, \beta, \sigma$ automorphisms of $\A$. Recall that a linear map $d: \A \to \A$ is called an {\bf $(\alpha, \beta)$-derivation of} $\A$ if it satisfies 
\[
d(xy) = d(x)\beta(y) + \alpha(x)d(y), \quad 
\forall \, x, y \in \A.
\]
A bilinear map $D: \A \times \A \to \A$ is said to be an {\bf $(\alpha, \beta)$-biderivation of} $\A$ if it is an $(\alpha, \beta)$-derivation in each argument. An {\bf $(\alpha, \beta)$-commuting} map
of $\A$ is a linear map $\Theta$ satisfying that $\Theta(x)\alpha(x) = \beta(x)\Theta(x)$, 
for all $x \in \A$. Clearly, every $\sigma$-derivation can be seen as an {\bf $(\alpha, \beta)$-derivation} with
$\alpha = \Id_\A$ and $\beta = \sigma$. The same can be said for $\sigma$-biderivations and $\sigma$-commuting maps.

It is straightforward to check that every $(\alpha, \beta)$-biderivation $D$ (respectively, $(\alpha, \beta)$-commuting map $\Theta$) of $\A$ gives rise to an $\alpha^{-1}\beta$-biderivation (respectively, $\alpha^{-1}\beta$-commuting map) by considering $\alpha^{-1}D$ (respectively, $\alpha^{-1}\Theta$). 
Accordingly, it is sufficient to restrict attention to $\sigma$-biderivations (respectively, $\sigma$-commuting maps).

In the last few years, many results on maps of triangular algebras have been extended to the setting of generalized matrix algebras (GMAs); see, for example, \cite{DW1, DW2, LW, XW} and references therein.
GMAs were introduced by Sands \cite{Sa} in the early 1970s. He ended up with these structures during his study of radicals of rings in Morita contexts. Specifically, a {\bf Morita context}
\[
(A, B, M, N, \Phi_{MN}, \Psi_{NM}),
\] 
consists of two $R$-algebras $A$ and $B$, two bimodules $_AM_B$ and $_BN_A$, and two bimodule homomorphisms $\Phi_{MN}: M \otimes_B N \to A$ and $\Psi_{NM}: N\otimes_A M \to B$ such that the following diagrams 
commute.

\[
\begin{CD}
M \otimes_B N \otimes_A M @>\Phi_{MN}\otimes I_M>> A \otimes_A M 
\\
@VVV @V\cong VV
\\
M \otimes_B B @>\cong>> M
\end{CD} 
\quad \quad \quad \, 
\begin{CD}
N \otimes_A M \otimes_B N
@>\Psi_{NM}\otimes I_N>> B \otimes_B N 
\\
@VVV @V \cong VV
\\
N \otimes_A A @>\cong>> N
\end{CD}
\]

\medskip

Let $(A, B, M, N, \Phi_{MN}, \Psi_{NM})$ be a Morita context, where at least one of the two bimodules is nonzero. The set 
\[
\mathcal{G} = \left(
\begin{array}{cc}
A & M  \\
N & B 
\end{array}
\right) = 
\left\{
\left(
\begin{array}{cc}
a & m  \\
n & b 
\end{array}
\right): \, a \in A, m\in M, n\in N, b \in B 
\right\},
\]

\noindent can be endowed with an $R$-algebra structure under the usual matrix operations. We will call $\mathcal{G}$ a {\bf generalized matrix algebra}. Note that any triangular algebra can be seen as a generalized matrix algebra with $N = 0$.
 
The extension of our results proceeds through the study of automorphisms of GMAs. However, the description of automorphisms of GMAs is still an open problem. (See \cite[Question 4.4]{LW})

\section*{Acknowledgements}

The second and third authors were supported by a grant from the Natural Sciences and Engineering Research 
Council (Canada). The first and third authors were supported by the Spanish MEC and Fondos FEDER 
jointly through project MTM2010-15223, and by the Junta de Andaluc\'ia (projects FQM-336, FQM2467 and FQM7156).
The third author thanks Professor Nantel Berger\'on, the Department of Mathematics at the University of Toronto and the Fields Institute for her visit from August 2012 to September 2013.

\end{document}